\documentclass[9pt]{amsart}
\usepackage{cancel}
\usepackage{amsmath,comment}
\usepackage{mathtools,bm}
\usepackage{amssymb}
\usepackage{amsfonts}
\usepackage{amsthm}
\usepackage{mathrsfs}
\usepackage{polynom}
\usepackage{mathbbol}
\usepackage{latexsym,amscd}
\usepackage[all]{xy}
\usepackage{color,graphicx}

\newtheorem{thm}{Theorem}
\newtheorem{prop}{Proposition}
\newtheorem{lem}{Lemma}
\newtheorem{conjecture}{Conjecture}
\numberwithin{equation}{section}
\numberwithin{prop}{section}
\numberwithin{lem}{section}
\numberwithin{thm}{section}
\newtheorem{cor}{Corollary}
\numberwithin{cor}{section}

\theoremstyle{definition}
\newtheorem{defn}{Definition}
\numberwithin{defn}{section}

\newtheorem{rem}{Remark}
\numberwithin{rem}{section}

\usepackage{breqn}

\newcommand{\ZZ}{\mathbb{Z}}
\newcommand{\CC}{\mathbb{C}}

\newcommand{\be}{\begin {equation}}
\newcommand{\ee}{\end{equation}}

\newcommand{\bee}{\begin {equation*}}
\newcommand{\eee}{\end{equation*}}

\newcommand{\nop}[1]{{}^{\scriptscriptstyle{\circ}}_{\scriptscriptstyle{\circ}}{#1}{}^{\scriptscriptstyle{\circ}}_{\scriptscriptstyle{\circ}}}

\def \F{\mathcal{F}}

\def \L{\Lambda}

\def \1{\mathbb{1}}
\def \V{\mathcal{V}}
\def \L{\mathcal{L}}

\usepackage{tikz-cd}

\begin{document}
\title{$S_3$-permutation orbifolds of Virasoro vertex algebras}
\author{Antun Milas, Michael Penn, Christopher Sadowski}
\address{Department of Mathematics and Statistics, SUNY-Albany}
\email{amilas@albany.edu}
\address{Mathematics Department, Randolph College}
\email{mpenn@randolphcollege.edu}
\address{Department of Mathematics and Computer Science, Ursinus College}
\email{csadowski@ursinus.edu}

\maketitle

\begin{abstract} 
In this paper, a continuation of \cite{MPS}, we investigate the $S_3$-orbifold subalgebra of $(\V_c)^{\otimes 3}$, that is, we consider the $S_3$-fixed point vertex subalgebra of the tensor product of three copies of the universal Virasoro vertex operator algebras $\V_c$. 
Our main result is construction of a minimal, strong set of generators of this subalgebra for any generic values of $c$. More precisely, we show that  
this vertex algebra is of type $(2,4,6^2,8^2,9,10^2,11,12^3)$.

We also investigate two prominent examples of simple $S_3$-orbifold algebras corresponding to central charges $c=\frac12$ (Ising model) and $c=-\frac{22}{5}$ (i.e. $(2,5)$-minimal model). 
We prove that the former is a new unitary $W$-algebra of type $(2,4,6,8)$ and the latter is isomorphic to the affine simple $W$-algebra of type $\frak{g}_2$ at non-admissible level $-\frac{19}{6}$. We also provide another version of this isomorphism using the affine $W$-algebra of type $\frak{g}_2$ coming from a subregular nilpotent element.

\end{abstract}

\section{Introduction}




In recent years there has been considerable attention given to the study of various aspects of permutation orbifold algebras of vertex algebras (see \cite{A, BDM,BHL,DLXY, DRX1,DRX2,DRX3, DRX4} and references therein, where the term ``orbifold" is used instead of ``orbifold (sub)algebra"). Although a lot is known about their representation theory, we know very little about the structure (e.g. their type) of permutation orbifold algebras even for familiar examples of vertex algebra such as Heisenberg, Virasoro, and affine vertex algebras. Apart from a few general results on orbifold algebras being finitely generated with respect to a reductive group \cite{L2,L3}, there is no known method based on invariant theory for finding a minimal generating set for an orbifold algebra. This is mainly due to non-trivial quantum corrections that depend on the vertex algebra being studied.
In our previous works, some with coauthors, \cite{MPS,MPSh,MP2,LMW}, we were able to construct minimal generating sets for low rank permutation orbifold algebras arising from 
the Heisenberg, free fermion, Virasoro, $\frak{sl}_2$ and $N=1,2$ superconformal algebras.
In particular, in \cite{MPS} we found a minimal generating set for the $3$-cycle permutation orbifold algebra associated to the Virasoro vertex algebra for every value of the central charge. 

To introduce the problem, we being by recalling our notation from \cite{MPS}. For a vertex algebra $V$, the $n$-fold tensor product will be denoted by $V^{\otimes n}:=V \otimes \cdots \otimes V$. The vector space  $V^{\otimes n}$ has a natural vertex operator algebra structure on which the symmetric group $S_n$ acts on $V^{\otimes n}$ by permuting tensor factors and thus $S_n \subset {\rm Aut}(V^{\otimes n})$. The ${S_n}$-invariant subalgebra of $V^{\otimes n}$, denoted by $(V^{\otimes n})^{S_n}$ we call the 
{\em $S_n$-{orbifold algebra}} of $V$. 
Throughout this work we only consider the Virasoro vertex algebra. We denote by $\V_c:=V_{Vir}(c,0)$ the universal Virasoro vertex operator algebra of central charge $c$.
It is (freely) generated by the weight $2$ conformal vector $\omega$ with vertex operator $Y(\omega,z)=L(z)=\sum_{n\in\ZZ}L(n)z^{-n-2}$ and the operator product expansion (OPE) 
\be
L(z)L(w) \sim \frac{\partial_w L(w)}{(z-w)}+\frac{2 L(w)}{(z-w)^2}+\frac{c/2}{(z-w)^4}.
\ee
 We denote by $\L_c:=L_{Vir}(c,0)$ its unique simple quotient. 
For $c=c_{p,q}:=1-\frac{6(p-q)^2}{pq}$, where where $p,q \geq 2$ are coprime integers (i.e. minimal models), we have $\V_c \neq \L_c$ and moreover $\L_c$ is a regular vertex algebra.
We already know that $Aut (\V_c^{\otimes n})=S_n$ (see for example \cite{MPS}) and therefore for any subgroup of $G \subset S_n$ we have a fixed point subalgebra $(\V_c^{\otimes n})^G$.  
The key question in this line of research is to describe the structure of $(\L_c^{\otimes n})^G$ using a minimal finite set of generators and to find possible isomorphisms with more familiar $W$-algebras. Some results in this direction were already given in \cite{MPS}.

In this paper we consider $S_3$-orbifold subalgebra of $(\V_c)^{\otimes 3}$ for a generic central charge $c$, and two distinguished simple orbifold algebras of central charge $c=\frac12$ and $c=-\frac{22}{5}$

This paper is organized as follows:
In Section 2 we setup the necessary notation and recall some basic facts about OPEs of fields of vertex algebras. We begin Section 3 by defining
a suitable large $c$ limit of the Virasoro algebra that we denote by $\V_\infty$. This vertex algebra and associated orbifold algebras can be used to obtain information about strong generators 
of orbifold algebras of $\V_c$ at least if $c$ is generic (see \cite{L3}). 
We first reduce generators to an infinite set of quadratic and cubic generators using standard methods of invariant theory. Then, in the most difficult part of the paper, we remove all 
but finitely many quadratic and cubic generators. Further reduction removes a few additional generators. The resulting generating set turns out to be minimal.
Our main result in this section can be summarized as:
\begin{thm} \label{main} For any generic $c$, including a suitably defined $c \to \infty$ limit, the $S_3$-orbifold subalgebra  $(\V_c^{\otimes n})^{S_3}$ is strongly generated by vectors of 
of weight $2,4,6,6,8,8,9,10,10,11,12,12,12$. Moreover, this is also a minimal generating set. 
\end{thm}

In Section 4 we switch our attention to a specific simple Virasoro vertex algebra, the unitary minimal model $\L_{\frac12}$ (Ising VOA), and its 
$S_3$-orbifold subalgebra $(\L_{\frac12}^{\otimes 3})^{S_3}$. This vertex algebra admits the well-known fermionic realization $\L_{\frac12} \hookrightarrow \mathcal{F}$, 
so that $\L_{\frac12}$ is the even part of the the rank one free fermion vertex algebra $\F$. Previously the first two authors and Wauchope investigated $(\mathcal{F}^{\otimes 3})^{S_3}$ and  determined its type \cite{MPW}. 
Although it seems now natural to consider a suitable $(\mathbb{Z}_2)^3$-orbifold of $(\mathcal{F}^{\otimes 3})^{S_3}$ to study $(\L_{\frac12}^{\otimes 3})^{S_3}$ we found this approach to be very cumbersome. For this reason, we do not use fermionic construction here and instead employ the null vectors of weight $6$ in $\L_{\frac 12}$ to obtain additional relations. This allows us to reduce several generators beyond the generic case. Our main result in this direction is
\begin{thm} The orbifold subalgebra $(\L_{\frac12}^{\otimes 3})^{S_3}$ is a (unitary) $W$-algebra of type $(2,4,6,8)$.
\end{thm}
We note that this vertex algebra is not isomorphic to a simple principal $W$-algebra of type $B_4$ (or $C_4$ by Feigin-Frenkel duality) nor the $\mathbb{Z}_2$-orbifold of the principal affine algebra of type $D_4$. Furthermore,  this algebra is not isomorphic to the parafermionic orbifold algebra $N_{10}(sl_2)^{\mathbb{Z}_2}$, which is also unitary of central charge $\frac32$.

Finally, in Section 5 we consider the $S_3$-orbifold algebra of $\L_{-22/5}$, the famous $(2,5)$-minimal model,  continuing our discussion from \cite{MPS}. 
Interestingly this orbifold algebra is related not to one but two affine vertex algebra of type $G_2$.
The next result was previously announced in \cite{MPS}. We denote by $W_k(\frak g_2,f_{prin})$ the simple principal affine $W$-algebra associated to $G_2$ at level $k$. 
\begin{thm} \label{25} We have an isomorphism 
$$(\L_{-\frac{22}{5}}^{\otimes 3})^{S_3} \cong W_{-\frac{19}{6}}(\frak g_2,f_{prin}).$$
In particular, the $W$-algebra on the right-hand side is regular (i.e. rational and lisse) \footnote{Using this isomorphism  we can also show that the affine $W$-algebra has precisely $24$ irreducible modules.}.
\end{thm}
Observe that the level $-\frac{19}{6}$ is not $\frak g_2$ admissible. This gives an example of a regular principal affine $W$-algebra outside admissible series.

We also have an isomorphism of simple VOAs $W_{-\frac{18}{5}}(\frak g_2,f_{prin}) \cong W_{-\frac{19}{6}}(\frak g_2,f_{prin})$, which is basically an instance of the Feigin-Frenkel duality of affine $W$-algebras. Moreover, using the construction of the subregular affine $W$-algebra of type $G_2$ (this algebra was also studied in J. Fasquel's PhD thesis  \cite{F}) we can easily see that $W_{-\frac{16}{5}}(g_2,f_{sub})=\L_{-22/5}^{\otimes 3}$ and therefore we can view $W_{-\frac{19}{6}}(\frak g_2,f_{prin})$ as an $S_3$-orbifold, i.e we get:
\begin{cor} 
$$W_{-\frac{16}{5}}(\frak g_2,f_{sub})^{S_3} \cong W_{-\frac{19}{6}}(\frak g_2,f_{prin}).$$
\end{cor}

\smallskip

{\bf Acknowledgments:} The first named author would like to thank J. Fasquel for discussion about affine $W$-algebras of type $\frak{g}_2$.

\section{Setup and Preliminary results}

As already mentioned in the introduction, $(\V_c,Y,\omega, \1)$ denotes the universal Virasoro vertex algebra with $\omega=L(-2){\1}$ and ${\1}$ denotes the vacuum vector.
Throughout the paper we will be working with $\V_c^{\otimes n}$ with $n=3$.
For convenience, we suppress the tensor product symbol and let  $$L_i(-m)\1:=\underbrace{\1 \otimes  \cdots  \1} _{(i-1)-{\rm factors}} \otimes L(-m) \1 \otimes \underbrace{ \1 \otimes \cdots \otimes \1}_{(n-i)-{\rm factors}} \in \V_c^{\otimes n},$$  such that
$\V_c^{\otimes n}=\langle L_1(-2)\1, \cdots ,L_n(-2)\1 \rangle$.
Thus $\omega=\omega_1+ \cdots + \omega_n$ is the total conformal vector in $\V_c^{\otimes n}$.
Using this notation, the natural action of $S_n$ on $\V_c^{\otimes n}$ is given by 
permuting tensor factors, that is 
\be
\sigma\cdot L_{i_1}(m_1)\cdots L_{i_k}(m_k)\1=L_{\sigma(i_1)}(m_1)\cdots L_{\sigma(i_k)}(m_k)\1,
\ee
for $1\leq i_j\leq n$, $m_j<-1$, and $\sigma\in S_n$. 

\begin{defn}
We say that $0 \neq v \in V$ is {\em primary} of conformal weight $r$ if $L(n)v=0$, $n \geq 1$ and $L(0)v=r v$.
In our work we sometimes consider vertex algebras $V$ strongly generated by the Virasoro vector and several primary fields
of conformal weight $r_1,...,r_k$. If so, using physics' terminology, we say that $V$ is a $W$-algebra of type
$(2,r_1,...,r_k)$, where $r_i$ can be repeated several times \footnote{If so we often write $r_i^n$, indicating that there are $n$ generators of conformal weight $r_i$.}. If generators of weight $r_i$ are not necessarily primary we still use the same notation but we omit calling $V$ a $W$-algebra.
\end{defn}

For a vertex algebra $V$ denote by ${\rm gr}(V)$ the associated graded Poisson algebra of $V$ \cite{Ar2, HLi}. We have a natural linear isomorphism 
\be\label{linearisom}
\V_c^{\otimes n}\cong {\rm gr}(\V_c^{\otimes n}) \cong \CC[x_i(m) | 1\leq i \leq n, m\geq 0]\ee
induced by $L_i(-m-2)\mapsto x_i(m)$ for $m\geq 0$. 
This algebra comes equipped with a derivation $\partial$ such that it is compatible with the translation operator in $\V_c^{\otimes n}$ given by $D(v)=v_{-2}\mathbb{1}$.
Then we have a standard result \cite{HLi,L2} (cf. \cite{Ar2}).
\begin{lem}\label{reconstruction}
Let $V$ be a vertex algebra with a ``good" $\mathbb{Z}_{\geq 0}$ filtration. If $\{\tilde{a}_i|i\in I\}$ generates $gr(V)$ then $\{a_i | i\in I\}$ strongly generates 
$V$, where $a_i$ and $\tilde{a}_i$ are related via the natural linear isomorphism described by the 
$\mathbb{Z}_{\geq 0}$ filtration.
\end{lem}

Because our main computation tool is the OPE package \cite{T}, throughout we will switch between working directly in the setting of the vertex operator algebra $\V_c^{\otimes n}$ and its copy inside $(\text{End }\V_c^{\otimes n})[[z,z^{-1}]]$ (fields) via the vertex operator map $$Y(\cdot,z):\V_c(n)\to (\text{End }\V^{\otimes n}_c(n))[[z,z^{-1}]],$$
i.e. we use the field-state correspondence.  Under this map we have 
\be\label{firstorbifoldgenerators}
u_k(m_1,\dots,m_k):=\sum_{i=1}^n L_i(-2-m_1)\cdots L_i(-2-m_k)\1,
\ee 
\be\begin{aligned}
U_k(m_1,\dots,m_k):=&Y(u_k(m_1,\dots,m_k),z)\\=&\frac{1}{(m_1-1)!}\cdots \frac{1}{(m_k-1)!}\sum_{i=1}^n\nop{\partial_z^{m_1}L_i(z)\cdots\partial_z^{m_k}L_i(z)},\end{aligned}\ee
where by $\nop{-}$ the normal ordered product and we will often suppress the formal variable $z$ and write $(\partial^m W):=\partial_z^m W(z)$, where $W(z)$ is any field.
Using this shorthand notation, we recall some basic facts about relations among normal ordered products (here $a,b$ and $c$ are arbitrary vectors in a vertex algebra):
\begin{equation}\label{nop1}
\nop{\nop{ab}c}  = \nop{abc} + \sum_{k \ge 0}\frac{1}{(k+1)!} \left( \nop{(\partial^{k+1}a)(b_{(k)}c)}+ \nop{(\partial^{k+1} b)(a_{(k)}c)}\right),
\end{equation}
\begin{equation}\label{nop2}
\nop{a_{(n)}(\nop{bc})} = \nop{(a_{(n)}b)c} + \nop{b(a_{(n)}c)}+ \sum_{k=1}^n {n \choose k} (a_{(n-k)}b)_{(k-1)}c,
\end{equation}
\begin{equation}\label{nop3}
(\nop{ab})_{(n)}c = \sum_{k \ge 0} \frac{1}{k!} \nop{(\partial^ka)(b_{(n+k)}c)} + \sum_{k \ge 0} b_{(n-k-1)}(a_{(k)}c).
\end{equation}
\color{black}

\section{The orbifold subalgebra $\left(\V_c^{\otimes 3}\right)^{S_3}$}

\subsection{Large $c$ limit of $\V_c$}
In order to analyze the orbifold subalgebra $\left(\V_c^{\otimes 3}\right)^{S_3}$ for generic values of $c$ we pass to the generalized free field limiting algebra as in \cite{L3}. 
From the OPE relations for $L(z)$, after rescaling $t:=\sqrt{c}$, $\alpha(z):=\frac{L(z)}{t}$, we obtain 
$$\alpha(z) \alpha(w) \sim \frac{1}{(z-w)^4}+\frac{2 \alpha(w)}{t (z-w)^2}+\frac{\partial \alpha(w)}{t (z-w)}.$$
Observe that for $t \to +\infty$ the limit is well-defined and we obtain an OPE algebra $\V_\infty:=\langle \alpha \rangle$, where
\begin{equation} \label{degen}
\alpha(z) \alpha(w) \sim \frac{1}{(z-w)^4}.
\end{equation}
Observe that $gr(\V_\infty) \cong gr(\V_c)$.
Then using the result of \cite{L3}, suitably adjusted, applied in our setup, we get:
\begin{prop} \label{linshaw} Let $u_i$ $i \in I$ be a strong set of generators of $(\V_\infty^{\otimes n})^{S_n}$ then for at most countably many values $c$ of the central charge, 
there is a strong generating set $t_i$, $ i \in I$ of $(\V_c ^{\otimes n})^{S_n}$ with $deg(t_i)=deg(u_i)$. 
\end{prop}

\begin{rem}
Due to order four pole in the OPE (\ref{degen}), $\V_\infty \neq \mathcal{H}$, where $\mathcal{H}$ is the rank one Heisenberg algebra (notice that the pairing between the modes of $T(z)$ is degenerate!).
For this reason we cannot simply use results from \cite{MPSh} on the structure of $\mathcal{H}(3)^{S_3}$ to analyze $(\V_\infty^{\otimes 3})^{S_3}$ .
\end{rem}

\subsection{Computations}
Define elements in $\V^{\otimes 3}_\infty=\langle \alpha_1,\alpha_2,\alpha_3 \rangle $:

\be\begin{aligned}
T_0&=\frac{1}{\sqrt{3}}(\alpha_1+ \alpha_2+ \alpha_3)\\
T_1&=\frac{1}{\sqrt{3}}(\alpha_1+\eta \alpha_2+\eta^2 \alpha_3)\\
T_2&=\frac{1}{\sqrt{3}}(\alpha_1+\eta^2 \alpha_2+\eta \alpha_3),
\end{aligned}\ee
where $\eta$ is a primitive third root of unity. Under this operation the algebra generated from $T_0$, $T_1$ and $T_2$ has the following nontrivial OPE
\be\begin{aligned}\label{OPE}
T_0(z)T_0(w)&\sim\frac{1}{(z-w)^4}\\
T_1(z)T_2(w)&\sim \frac{1}{(z-w)^4}.\end{aligned}\ee
Further, as this algebra is the limit of a Virasoro vertex operator algebra, we can define the weight of an element from our algebra as the conformal weight of one of its preimages under the limiting procedure. One can easily check that this implies that $\text{wt }T_i=2$ and this is well-defined. \color{black}
Now we set $\mathcal{A}=\left<T_0,T_1,T_2\right> $ and thus we have 
\be \V_\infty^{\otimes 3}\cong \mathcal{A}\ee
and 
\be (\V_\infty^{\otimes 3})^{S_3} \cong \mathcal{A}^{S_3}. \ee
From (\ref{OPE}) we have for $k \ge 0$
\begin{equation}
(\partial^m T_1)_{(k)}(\partial^n T_2) = \frac{(-1)^m(m+n+3)!}{3!}\delta_{k,m+n+3}\1
\end{equation}
and
\begin{equation}
(\partial^m T_2)_{(k)}(\partial^n T_1) = \frac{(-1)^m(m+n+3)!}{3!}\delta_{k,m+n+3}\1
\end{equation}

\color{black}
Now, following \cite{MPSh} we set 
\be\begin{aligned}W_{m,n}&=\nop{(\partial^{m} T_1)(\partial^{n} T_2)}+\nop{(\partial^{n} T_1)(\partial^{m} T_2)}\\
C_{\ell,m,n}&=\nop{(\partial^{\ell} T_1)(\partial^{m} T_1)(\partial^{n} T_1)}+\nop{(\partial^{\ell} T_2)(\partial^{m} T_2)(\partial^{n} T_2)}\end{aligned}
\ee
and it is clear that the orbifold subalgebra $\mathcal{A}^{S_3}$ is strongly generated by the fields $T_0$, $W_{m,n}$, and $C_{\ell,m,n}$ for $\ell,m,n\geq 0$.  Generators $W_{m,n}$ are elements of degree two in the associated graded algebra of ${\rm gr}(\mathcal{A}^{S_3})$ so these generators are called quadratic, likewise $C_{\ell,m,n}$ are called cubic. Using quantum corrections we will show that $\mathcal{A}^{S_3}$ has a finite strong set of generators (this is not the case with ${\rm gr}(\mathcal{A}^{S_3})$!). 
Let us fix our generating set: 
\begin{equation}
\mathcal{G}= \{ T_0, W_{m,n}, C_{\ell,m,n} : \ell, m,n \geq 0 \}.
\end{equation}
These generators are ordered according to the weight defined earlier. If $a \in \mathcal{G}$ can be written as a linear combination of $\partial^i b$, $i \geq 1$ and $\nop{ a_1 \cdots a_k}$ where ${\rm wt}(a_i) < {\rm wt}(b)$, then we say that $a$ can be expressed in terms of (generators) 
of lower weight. If so, then $\mathcal{G}$ can be reduced to $\mathcal{G} \setminus a$.
\color{black}
Let us explain the strategy of our proof of Theorem \ref{main}, which is split into two parts.
In the first part right below, we show that we can remove all but finitely many cubic generators from $\mathcal{G}$ (and still have a strong generating set).
In the second part we consider various quadratic relations together with the remaining cubic generators and show that all but finitely many quadratic generators 
remain in $\mathcal{G}$.

\subsection{Cubic generators}
As the first reduction, using techniques involving the translation operator similar to \cite{MPS} and \cite{L2} we see that our orbifold algebra is in fact strongly generated by $T_0$, $W_{2m,0}$, and $C_{m,n,0}$ for $m\geq n\geq 0$ so $\mathcal{G}$ can be replaced by this set. 
In particular, we have
$$C_{\ell,m,n}=(-1)^n\sum_{k=0}^n {n \choose k} C_{m+k,n+\ell-k,0}.$$

 Further reductions required a bit more work and our major calculation tool will be the identity
\be\begin{aligned}
\nop{W_{m_1,m_2}C_{n_1,n_2,n_3}}=&\frac{1}{6}\left(\left(\frac{(-1)^{m_1}}{m_1+n_1+4}+\frac{(-1)^{m_2}}{m_2+n_1+4}\right)C_{m_1+m_2+n_1+4,n_2,n_3}\right.\\
&+\left(\frac{(-1)^{m_1}}{m_1+n_2+4}+\frac{(-1)^{m_2}}{m_2+n_2+4}\right)C_{m_1+m_2+n_2+4,n_1,n_3}\\
&+\left.\left(\frac{(-1)^{m_1}}{m_1+n_3+4}+\frac{(-1)^{m_2}}{m_2+n_3+4}\right)C_{m_1+m_2+n_3+4,n_1,n_2}\right)\\
&+\Psi,
\end{aligned}\ee
where $\Psi$ is a linear combination of vectors that are degree 5 in the associated graded algebra. Using this we can construct a parameterized family of expressions 
\be\begin{aligned}\label{masterrel}
R_1(\mathbf{a},\mathbf{m})&=(a_1+a_2+a_3)\nop{W_{m_4,m_5}C_{m_1,m_2,m_3}}+(a_1+a_3+a_5)\nop{W_{m_3,m_5}C_{m_1,m_2,m_4}}\\
&-(a_1+a_2+a_3+a_4+a_5)\nop{W_{m_3,m_4}C_{m_1,m_2,m_5}}+a_5\nop{W_{m_2,m_4}C_{m_1,m_3,m_5}}\\
&-(a_1+a_4+a_5)\nop{W_{m_2,m_5}C_{m_1,m_3,m_4}}-(a_1+a_2+a_3)\nop{W_{m_1,m_5}C_{m_2,m_3,m_4}}\\
&+a_4\nop{W_{m_2,m_3}C_{m_1,m_4,m_5}}+a_3\nop{W_{m_1,m_4}C_{m_2,m_3,m_5}}\\
&+a_1\nop{W_{m_1,m_2}C_{m_3,m_4,m_5}}+a_2\nop{W_{m_1,m_3}C_{m_2,m_4,m_5}},
\end{aligned}\ee
where $\mathbf{a}=(a_1,\dots,a_5)$ and $\mathbf{m}=(m_1,\dots,m_5)$. The fact that $R(\mathbf{a},\mathbf{m})=0$  in the associated graded algebras allows us to create certain quantum corrections in order to write the $C_{a_1,a_2,a_3}$ generators in terms of lower weight generators. \color{black} Of particular interest will be the four relations $R_1(\mathbf{a},m,n,4,1,0)$, $R_1(\mathbf{a},m,n,3,2,0)$, $R_1(\mathbf{a},m,n,3,1,1)$, and $R_1(\mathbf{a},m,n,2,2,1)$ for $m>n>4$. The first two of these produce five linearly independent relations (built from choices of the $a_i$), while the last two produce three linearly independent relations. In total, we have 16 total relations at weight $m+n+12$. This is most fortunate as the expansion of these relations only consider sixteen cubic generating fields $C_{m+i,n+9-i,0}$ for $0\leq i\leq 9$, $C_{m+n+4+i,5-i,0}$ for $0\leq i\leq 3$, $C_{m+n,9,0}$, and $C_{m+n+9,0,0}$ and thus these relations may be used to write these generators in term of lower weight terms. The result is that all cubic generating fields of weight 23 and higher can be written in terms of cubic generators between weights 6 and 22. What remains is to minimize this fairly large set. This process can be achieved by carefully taking linear combination of our relations \eqref{masterrel}. For example, if we set $R_2(\mathbf{m})=R_1(1,0,0,0,0,\mathbf{m})$, we can find $b_1$,..., $b_8$ such that 
$$\begin{aligned}C_{16,0,0}=&b_1R_2(10,2,0,0,0)+b_2R_2(9,2,1,0,0)+b_3R_2(8,3,1,0,0)+b_4R_2(7,3,2,0,0)\\
+&b_5R_2(6,5,1,0,0)+b_6R_2(5,4,2,1,0)+b_7R_2(4,3,2,2,1)+b_8R_2(4,4,4,0,0).\end{aligned}$$
These coefficients are quite unruly, for example
\begin{align*}
& b_1=\frac{1790484010217545392288}{168520823757097513517}, b_2=\frac{1795809487559936088240}{168520823757097513517}, \ldots  \\
& b_8=-\frac{1464894501954686124462}{168520823757097513517}.
\end{align*}
Similar equations can be constructed for $C_{14,2,0}$,...,$C_{8,8,0}$, eliminating the need for these generators as well. Finally, since $\partial C_{15,0,0}=C_{16,0,0}+2C_{15,1,0}$, $C_{15,1,0}$ is not needed either. Together, all weight 22 cubic generators can be written in terms of lower weight fields. 

Conformal weight 13 is the lowest weight at which we can remove all of the cubic generators, namely the fields $C_{7,0,0}$, $C_{6,1,0}$, $C_{5,2,0}$, and $C_{4,3,0}$. The field $C_{6,1,0}$ can be removed only using the translation operator as described above, so we focus on the remaining three. In order to eliminate these weight 13 cubic generators we consider the following system of three equations
$$\begin{aligned}3C_{7,0,0}+42C_{4,3,0}-28\partial C_{3,3,0}=&14\partial C_{6,0,0}-252\partial^2C_{3,2,0}+42\partial^3C_{2,2,0}-77\partial^4C_{3,0,0}\\&+84\partial^5C_{2,0,0}-6\partial^7 C_{0,0,0},\\27C_{5,2,0}+81C_{4,3,0}-24\partial C_{3,3,0}=&12\partial C_{6,0,0}-351\partial^2C_{3,2,0}+81\partial^3C_{2,2,0}-96\partial^4C_{3,0,0}\\&+108\partial^5C_{2,0,0}-8\partial^7 C_{0,0,0}\\30C_{4,3,0}+12 C_{5,2,0}+13C_{7,0,0}=&R_2(2,1,0,0,0).\end{aligned}$$
Next, one can check that the determinant of the left-hand side is nonzero and thus 
we can write $C_{7,0,0}$, $C_{5,2,0}$, and $C_{4,3,0}$ in terms of lower weight generators. \color{black}

\color{black}

The lowest weight ``quantum correction'' relation occurs at weight 12 where, prior to this final reduction, the necessary generators are $C_{6,0,0}$, $C_{4,2,0}$, and $C_{3,3,0}$ -- $C_{5,1,0}$ can immediately be removed by our previous discussion. Next, we have
$$30C_{4,2,0}=-79C_{6,0,0}+R_2(1,1,0,0,0),$$
leaving only the for $C_{6,0,0}$ and $C_{3,3,0}$, which may  not be removed. Interestingly, there is a nontrivial relation involving cubic generators at weight 12 but it only involves terms that have been differentiated. Namely, we have
$$9\partial^2C_{4,0,0}-18\partial^2C_{2,2,0}-24\partial^3C_{3,0,0}+18\partial^4C_{2,0,0}-\partial^6C_{0,0,0}=0.$$
Next, there are no ``quantum correction'' relations at weight 11 but we do have a nontrivial relation involving the translation operator,
$$30C_{3,2,0}=3C_{5,0,0}-15\partial^2C_{3,0,0}+15\partial^3C_{2,0,0}-\partial^5C_{0,0,0},$$
meaning the only weight 11 generator needed is $C_{5,0,0}$. Similarly, at weight 10, we only need $C_{4,0,0}$ because
$$18C_{2,2,0}=9C_{4,0,0}-24\partial C_{3,0,0}+18\partial^2 C_{2,0,0}-\partial^4C_{0,0,0}.$$
After all of this, and other similar calculations, the only cubic generators that are required in $\mathcal{G}$ are $C_{0,0,0}$, $C_{m,0,0}$ for $2\leq m\leq 6$, and $C_{3,3,0}$.

\subsection{Quadratic generators} Now we move to the quadratic terms $W_{2m,0}$ for $m \ge 0$. We first establish some basic facts about $W_{a,b}$.
Using the fact that 
\begin{equation}
\partial W_{a,b} = W_{a+1,b} + W_{a,b+1}
\end{equation}
we have using Binomial Theorem
\begin{equation}\label{gen-der}
\partial^k W_{a,0} = \sum_{j=0}^k {k \choose j} W_{a+k-j,j}
\end{equation}
and we rewrite $W_{a,b}$ as
\begin{equation}\label{der-rewrite}
W_{a,b}  = \sum_{i=0}^b(-1)^{b-i} {b \choose i} \partial^i W_{a+b-i,0}
\end{equation}
We need the following lemmas:
\begin{lem}
For $p_1,p_2,m_1,m_2 \ge 0$ we have
\begin{equation}\label{0-prod}
(W_{p_1,p_2})_{(0)} W_{m_1,m_2} = \frac{(-1)^{p_1} + (-1)^{p_2}}{3!}W_{p_1+p_2+m_1+3,m_2} + \frac{(-1)^{p_1} + (-1)^{p_2}}{3!}W_{p_1+p_2+m_2+3,m_1} 
\end{equation}
and
\begin{align}\label{1-prod}
(W_{p_1,p_2})_{(1)} W_{m_1,m_2} &= \frac{(-1)^{p_1}(p_1+m_1+3) + (-1)^{p_2}(p_2+m_1+3)}{3!}W_{p_1+p_2+m_1+2,m_2} \\
&+ \frac{(-1)^{p_1}(p_1+m_2+3) + (-1)^{p_2}(p_2+m_2+3)}{3!}W_{p_1+p_2+m_2+2,m_1} 
\end{align}
\end{lem}
\begin{proof}
This follows immediately by direct computation with (\ref{nop1}) - (\ref{nop3}).
\end{proof}
Our main tool will be the operator $(W_{0,0})_{(1)}$. In particular, by direct computation using (\ref{1-prod}) we have that 
\begin{equation}\label{old-00-1-action}
(W_{0,0})_{(1)}W_{a,0} = \frac{a+3}{3}W_{a+2,0} +W_{2,a} = \frac{a+3}{3}W_{a+2,0} +W_{a,2} 
\end{equation}
Now, using (\ref{der-rewrite}) we have
\begin{equation}
W_{a,2} = W_{a+2,0} - 2\partial W_{a+1,0} + \partial^2 W_{a,0}
\end{equation}
and so we rewrite  (\ref{old-00-1-action}) as:
\begin{align}\label{00-1-action}
(W_{0,0})_{(1)}W_{a,0} &= \frac{a+3}{3}W_{a+2,0} + W_{a+2,0} - 2\partial W_{a+1,0} + \partial^2 W_{a,0}\nonumber \\
&=  \frac{a+6}{3}W_{a+2,0}  - 2\partial W_{a+1,0} + \partial^2 W_{a,0}
\end{align}

\begin{lem}\label{1-prod-k}
For $a\ge 0$ and $k \ge 2$ we have
\begin{equation}\label{1-prod-on-k-der}
(W_{0,0})_{(1)}(\partial^k W_{a,0}) = \sum_{j=0}^k\sum_{i=2}^j \frac{1}{3}{k \choose j}(-1)^{j-i}\left( (a+k-j+3){j \choose i} + (j+3) {j+2 \choose i}\right)\partial^i W_{a+k+2-i,0}
\end{equation}
\end{lem}
\begin{proof}
We have, by application of (\ref{gen-der}) and then (\ref{1-prod}), 
\begin{align*}
(W_{0,0})_{(1)}(\partial^k W_{a,0}) &= (W_{0,0})_{(1)}\left(\sum_{j=0}^k {k \choose j} W_{a+k-j,j}\right)\\
&= \sum_{j=0}{k \choose j}\frac{1}{3} (a+k-j+3)W_{a+k-j+2,j} + \sum_{j=0}^k {k \choose j} \frac{1}{3} (j+3) W_{a+k-j,j+2}.
\end{align*}
Now, using (\ref{der-rewrite}) we have
\begin{align*}
(W_{0,0})_{(1)}(\partial^k W_{a,0}) &=\sum_{j=0}^k {k \choose j} \frac{1}{3} (a+k-j+3)\sum_{i=0}^j{j \choose i} (-1)^{j-i}\partial^i W_{a+k+2-i,0} \\
&+ \sum_{j=o}^k {k \choose j} \frac{1}{3} (j+3) \sum_{i=0}^{j+2} {j+2 \choose i}(-1)^{j+2-i}\partial^i W_{a+k+2-i,0}\\
&=  \sum_{j=0}^k\sum_{i=0}^j \frac{1}{3}{k \choose j}(-1)^{j-i}\left( (a+k-j+3){j \choose i} + (j+3) {j+2 \choose i}\right)\partial^i W_{a+k+2-i,0}
\end{align*}
We note that when $i=0$ we have 
\begin{equation}
\sum_{j=0}^k \frac{1}{3}{k \choose j} (-1)^j (a+k+6) W_{a+k+2,0} = 0
\end{equation}
since 
\begin{equation}
\sum_{j=0}^k {k \choose j} (-1)^j = 0.
\end{equation}
Moreover, when $i=1$ we have
\begin{align*}
\sum_{j=0}^k &\frac{1}{3}{k \choose j} (-1)^{j-1}\left((a+k-j+3)j + (j+3)(j+2)\right)\partial W_{a+k+1,0} \\
&= \sum_{j=0}^k \frac{1}{3}{k \choose j} (-1)^{j-1}\left((a+k+8)j+6\right)\partial W_{a+k+1,0}  .
\end{align*}
We note that 
\begin{equation}
\sum_{j=0}^k (-1)^j {k \choose j}((a+k+8)j+6) = 0
\end{equation}
(here we use the fact that $\sum_{j=0}^k (-1)^j {k \choose j} P(j) = 0$ when $P(j)$ is a polynomial of degree $\le k$). Thus our claim is proved.
\end{proof}
We have thus shown that when $(W_{0,0})_{(1)}$ acts on a second or higher derivative of $W_{a,0}$ it introduces no $0$-th and $1$-st derivatives of terms of the form $W_{n,0}$.
\begin{lem}\label{W001-product-rule}
\begin{equation}
(W_{0,0})_{(1)} \nop{W_{a,b} W_{c,d}} = \nop{ ((W_{0,0})_{(1)}W_{a,b})W_{c,d}} + \nop{ W_{a,b}((W_{0,0})_{(1)}W_{c,d})} 
\end{equation}
\end{lem}
\begin{proof}
We use (\ref{nop2}) to obtain 
\begin{equation*}
(W_{0,0})_{(1)} \nop{W_{a,b} W_{c,d}} = \nop{ ((W_{0,0})_{(1)}W_{a,b})W_{c,d}} + \nop{ W_{a,b}((W_{0,0})_{(1)}W_{c,d})}  + ((W_{0,0})_{(0)}W_{a,b})_{(0)}W_{c,d}
\end{equation*}
Using (\ref{0-prod}) twice we have that
\begin{align*}
((W_{0,0})_{(0)}W_{a,b})_{(0)}W_{c,d} &= \frac{1}{3}(W_{a+3,b} + W_{a,b+3})_{(0)}W_{c,d}\\
&= \frac{1}{3}(W_{a+3,b})_{(0)}W_{c,d} + \frac{1}{3}(W_{a,b+3})_{(0)}W_{c,d}\\
&= \frac{1}{3}\left(\frac{(-1)^{a+3} + (-1)^b}{3!}W_{a+3+b+c+3,d} + \frac{(-1)^{a+3} + (-1)^b}{3!}W_{a+3+b+d+3,c} \right)\\
&+\frac{1}{3}\left( \frac{(-1)^{a} + (-1)^{b+3}}{3!}W_{a+b+3+c+3,d} +  \frac{(-1)^{a} + (-1)^{b+3}}{3!}W_{a+b+3+d+3,c} \right) \\
&= 0
\end{align*}
thus proving our claim.
\end{proof}

 First, we note that the generators $W_{0,0},W_{2,0},\dots, W_{8,0}$ cannot be removed as there are no relations which allow us to rewrite them as normally ordered polynomials of our remaining cubic generators and normally ordered polynomials of lower weight terms.  \color{black} The following terms, however, can be removed:
\begin{align*}
W_{10,0} &= \frac{17360}{6777}\nop{C_{0,0,0} \partial C_{1,0,0}}-\frac{2240}{20331} \nop{\partial^2C_{0,0,0}C_{0,0,0}}-\frac{560}{251}  \nop{C_{0,0,0}C_{2,0,0}}-\frac{1120}{251} \nop{C_{1,0,0}C_{1,0,0}}\\
&+\frac{10360 }{2259}\nop{\partial W_{0,0},\partial^2W_{3,0}}+\frac{86744668 }{101655}\nop{\partial^2 W_{1,0},\partial W_{2,0}}-\frac{50002820}{20331} \nop{\partial^2W_{0,0}\partial W_{3,0}}\\
&+\frac{3920 }{2259}\nop{W_{0,0}W_{0,0}\partial W_{1,0}}-\frac{350}{251} \nop{W_{0,0}\partial W_{5,0}}+\frac{458080 }{6777}\nop{\partial W_{0,0}W_{0,0},W_{1,0}}\\
&+\frac{81200 }{6777}\nop{\partial W_{0,0}\partial W_{0,0}W_{0,0}}+\frac{336}{251} \nop{\partial W_{0,0}W_{5,0}}+\frac{173270936 }{101655}\nop{\partial W_{2,0}W_{3,0}}-\frac{86874868 }{33885}\nop{\partial W_{2,0}\partial W_{2,0}}\\
&-\frac{10992884 }{101655}\nop{W_{2,0}\partial W_{3,0}}+\frac{560 }{2259}\nop{\partial^2 W_{0,0}W_{0,0},W_{0,0}}+\frac{24975370 }{6777}\nop{\partial^2W_{0,0},\partial^2 W_{2,0}}\\
&+\frac{5307442}{33885} \nop{\partial^2W_{2,0},W_{2,0}}-\frac{12484430 }{20331}\nop{\partial^4W_{0,0}\partial^2W_{0,0}}-\frac{560}{251} \nop{W_{0,0}W_{0,0},W_{2,0}}\\
&+\frac{3290 }{2259}\nop{W_{0,0},W_{6,0}}+\frac{2240}{251} \nop{W_{2,0},W_{4,0}}+\frac{40600 }{6777}\nop{W_{3,0},W_{3,0}}-\frac{1256080 }{6777}\nop{W_{0,0}W_{1,0},W_{1,0}}\\
&-\frac{5380942 }{101655}\nop{\partial^3 W_{1,0},W_{2,0}}+\frac{34423 }{4518}\partial^2W_{8,0}
\end{align*}
\begin{align*}
W_{12,0} &= -\frac{149957732453760}{5282271130981} \nop{C_{0,0,0}C_{4,0,0}}+\frac{29449282789080 }{5282271130981}\nop{C_{0,0,0}\partial C_{3,0,0}}\\
&+\frac{248361731055900 }{5282271130981}\nop{C_{0,0,0}\partial^2 C_{2,0,0}}+\frac{83092353248520 }{5282271130981}\nop{C_{1,0,0}C_{3,0,0}}\\
&+\frac{110166066382680 }{5282271130981}\nop{C_{2,0,0}C_{2,0,0}}+\frac{25002706049441505 }{21129084523924}\nop{W_{0,0}W_{0,0}\partial^2 W_{2,0}}\\
&+\frac{144692176980100662 }{26411355654905}\nop{W_{0,0}W_{1,0}\partial W_{2,0}}+\frac{964865299915440 }{5282271130981}\nop{W_{0,0}W_{2,0}W_{2,0}}\\
&+\frac{96307517079469080 }{5282271130981}\nop{W_{0,0}\partial^2 W_{1,0}W_{1,0}}+\frac{205714934779662 }{5282271130981}\nop{W_{2,0}\partial W_{5,0}}\\
&+\frac{30867843011111491}{132056778274525} \nop{W_{3,0}W_{5,0}}+\frac{905465924348909 \nop{W_{3,0}\partial W_{4,0}}}{52822711309810}\\
&+\frac{118788131274120 }{5282271130981}\nop{W_{4,0}W_{4,0}}+\frac{404810764267140 }{5282271130981}\nop{\partial C_{0,0,0} \partial^2 C_{1,0,0}}\\
&+\frac{317276622972540 }{5282271130981}\nop{\partial C_{1,0,0}C_{2,0,0}}+\frac{392721757761984981 }{26411355654905}\nop{\partial W_{0,0}W_{0,0}W_{3,0}}\\
&+\frac{25992477708120 }{406328548537}\nop{\partial W_{0,0}\partial W_{0,0}W_{2,0}}+\frac{385990604483438691 }{26411355654905}\nop{\partial W_{0,0}\partial W_{0,0} \partial W_{1,0}}\\
&+\frac{15048061294365 }{5282271130981}\nop{\partial W_{0,0}W_{7,0}}+\frac{1102902314088960 }{5282271130981}\nop{\partial W_{3,0}W_{4,0}}\\
&+\frac{3287216023258860 }{5282271130981}\nop{\partial^2 W_{0,0}W_{0,0}W_{2,0}}+\frac{240861688627154562 }{26411355654905}\nop{\partial^2 W_{0,0}W_{0,0}\partial W_{1,0}}\\
&+\frac{925716380630407641}{52822711309810} \nop{\partial^2 W_{0,0}\partial W_{0,0}\partial W_{0,0}}+\frac{205032882406275 }{10564542261962}\nop{\partial^2 W_{0,0}W_{6,0}}\\
&+\frac{21869614107495 }{406328548537}\nop{\partial^2 W_{2,0}\partial W_{3,0}}+\frac{203770339785809577 }{52822711309810}\nop{\partial^3 W_{0,0}\partial W_{0,0}W_{0,0}}\\
&-\frac{32823424473840}{406328548537} \nop{C_{1,0,0}\partial C_{2,0,0}}-\frac{4281461005800}{406328548537} \nop{W_{0,0}W_{8,0}}\\
&-\frac{47308239047685 }{5282271130981}\nop{C_{0,0,0}\partial^3 C_{1,0,0}}-\frac{149957732453760 }{5282271130981}\nop{W_{0,0}W_{0,0}W_{4,0}}\\
&-\frac{7718192159554800 }{5282271130981}\nop{W_{0,0}\partial W_{1,0}W_{2,0}}-\frac{92540076073830 }{5282271130981}\nop{W_{2,0}W_{6,0}}\\
&-\frac{157046463119700 }{5282271130981}\nop{\partial C_{1,0,0}\partial C_{1,0,0}}-\frac{113139651104865 }{5282271130981}\nop{\partial W_{0,0}\partial W_{6,0}}\\
&-\frac{1476111467740140 }{5282271130981}\nop{\partial W_{2,0}W_{5,0}}-\frac{718283710890900 }{5282271130981}\nop{\partial W_{2,0}\partial^2 W_{3,0}}\\
&-\frac{923178502476615 }{5282271130981}\nop{\partial W_{3,0}\partial W_{3,0}}-\frac{900339366562020 }{5282271130981}\nop{\partial^2 C_{1,0,0}C_{1,0,0}}\\
&-\frac{1252803312074160 }{5282271130981}\nop{\partial^2 W_{2,0}W_{4,0}}-\frac{919748819913000 }{5282271130981}\nop{\partial^2 W_{3,0}W_{3,0}}\\
&-\frac{26848347447845}{5282271130981} \nop{\partial^4 C_{0,0,0},C_{0,0,0}}-\frac{7944187809494475 }{10564542261962}\nop{W_{0,0}W_{0,0}\partial W_{3,0}}\\
&-\frac{845384942454129}{10564542261962} \nop{\partial^2 W_{0,0}\partial W_{5,0}}-\frac{319311041738535 }{21129084523924}\nop{\partial^2 W_{2,0}\partial^2 W_{2,0}}\\
&-\frac{47257775301300762 }{26411355654905}\nop{W_{0,0}W_{1,0}W_{3,0}}-\frac{228752216559403812 }{26411355654905}\nop{W_{0,0}\partial W_{1,0} \partial W_{1,0}}\\
&-\frac{625388315162031531}{26411355654905} \nop{\partial W_{0,0}W_{0,0}\partial W_{2,0}}-\frac{2624199202543291164 }{26411355654905}\nop{\partial W_{0,0}\partial W_{1,0}W_{1,0}}\\
&-\frac{62850174015927078 }{26411355654905}\nop{\partial^2 W_{0,0}\partial^2 W_{0,0}W_{0,0}}-\frac{260626674972175677 }{26411355654905}\nop{\partial^3 W_{0,0}W_{0,0}W_{1,0}}\\
&-\frac{8514499500637515 }{42258169047848}\nop{\partial^4 W_{0,0}W_{0,0}W_{0,0}},
\end{align*}
and $W_{14,0}$, for which we do not display an explicit formula for the sake of brevity.
We note that $W_{14,0}$ can be written purely in terms of generators of the form $W_{2m,0}$ for $0 \le m \le 6$. Importantly, we note that any $\partial^k W_{2m,0}$ term which is not part of a product is a second derivative or higher ($k \ge 2$). 

Suppose now that we have rewritten $W_{a,0}$, $a$ even, as a normally ordered polynomial of lower weight quadratic terms and their derivatives \color{black}, where any $\partial^k W_{2m,0}$ term which is not part of a product is a second derivative or higher ($k \ge 2$). We then have, using (\ref{00-1-action}), 
\begin{align*}
W_{a+2,0} = \frac{3}{a+6}\left( (W_{0,0})_{(1)}W_{a,0} +2\partial W_{a+1,0} - \partial^2 W_{a,0}\right)
\end{align*}
First we examine $(W_{0,0})_{(1)}W_{a,0} $. We note that using Lemma \ref{W001-product-rule}, any normally ordered product of two or more terms of the form $W_{k,0}$ and their derivatives will itself become a sum of such products. Using Lemma \ref{1-prod-k} we see that with application of $(W_{0,0})_{(1)}$ the only $\partial^k W_{2m,0}$ which $(W_{0,0})_{(1)}$ introduces that is not part of a product is a second derivative or higher and no first and $0$-th derivatives are introduced. Next, we note that $\partial W_{a+1,0}$ can be rewritten in terms of $\partial^2 W_{a,0}$ and a normally ordered polynomial of terms of lower weight and their derivatives. Lastly, we note that $\partial^2 W_{a,0}$ is of the correct form. Thus, we have eliminated all quadratic generators $W_{2m,0}$ for $m \ge 5$.

To summarize, cubic generators  $C_{0,0,0}$, $C_{m,0,0}$ for $2\leq m\leq 6$, $C_{3,3,0}$,  quadratic generators $W_{0,0},W_{2,0},\dots, W_{8,0}$, and $T_0$ form a strong generating set of $\mathcal{A}^{S_3}$.

Equipped with a strong generating set for $(\V_\infty^{\otimes 3})^{S_3}$ we also obtain a strong generating set for $(\V_c^{\otimes 3})^{S_3}$ for any generic value using Proposition \ref{linshaw}. 
To see that this is in fact a minimal generating set it is sufficient to compare the character ${\rm ch}[\mathcal{A}^{S_3}](q)$  \cite{MPS} and 
the "free" character coming from the obtained generators $$\frac{1}{(q^2;q)_\infty (q^4;q)_\infty (q^6;q)^2_\infty (q^8;q)^2_\infty (q^9;q)_\infty (q^{10};q)^2_\infty (q^{11};q)_\infty (q^{12};q)^3_\infty }.$$ These two $q$-series agree $O(q^{13})$ so no generator up to weight $12$ can be removed 
from the generating set so constructed a minimal set. This finishes the proof of Theorem \ref{main} in the introduction.

\begin{rem} We expect that for all generic values the generators of conformal weight $>2$ can be replaced with primary vectors.
\end{rem}

\color{black}
\section{The simple orbifold $c=\frac{1}{2}$}
In this section we consider the special case when the initial central charge is $\frac{1}{2}$, and thus the final central charge is $\frac{3}{2}$. This has a nice connection to the universal even spin  VOAs.
We now work inside $\V_{\frac12}^{\otimes 3}$.
We begin with the fields
\be\label{chalfsing}
v_i=\nop{L_iL_iL_i}+\frac{93}{64}\nop{(\partial L_i)(\partial L_i)}-\frac{33}{16}\nop{(\partial^2 L_i)L_i}-\frac{9}{128}\partial^4 L_i\ee
for $i=1,2,3$ which are each singular in $\mathcal{V}_{\frac{1}{2}}^{\otimes 3}$ -- as they are each singular in their appropriate copies of $\V_{\frac{1}{2}}$. From here we define an alternative to the standard generating set of $\mathcal{V}_{\frac{1}{2}}^{\otimes 3}$ which diagonalizes the action of $(1 2 3)\in S_3$
\be\begin{aligned}\label{z3gens}
L&=\frac{1}{\sqrt{3}}(L_1+L_2+L_3)\\
U_1&=\frac{1}{\sqrt{3}}(L_1+\eta L_2+\eta^2L_3)\\
U_2&=\frac{1}{\sqrt{3}}(L_1+\eta^2 L_2+\eta L_3),\end{aligned}\ee
where $\eta$ is a primitive third root of unity. Next, we set
\be\label{ws} W_{m+4}=\nop{(\partial^m U_1)U_2}+(-1)^m\nop{(\partial^m U_2)U_1}\ee
and 
\be\label{cs} C_{m+6}^{\pm}=\nop{(\partial^m U_1)U_1U_1}\pm\nop{(\partial^m U_2)U_2U_2}\ee
and by \cite{MPS} we know that $\left(\mathcal{V}_{\frac{1}{2}}^{\otimes 3}\right)^{\mathbb{Z}_3}$ is strongly generated by $W_{m_1}$,$C^{+}_{m_2}$, and $C^{-}_{m_3}$ for $m_1\in\{4,5,6,7,8,9,10\}$, $m_2\in\{6,8,9,10\}$, and $m_3\in\{6,8,9\}.$ Now we transport the singular vectors \eqref{chalfsing} into $\left(\mathcal{V}_{\frac{1}{2}}^{\otimes 3}\right)^{\mathbb{Z}_3}$ by defining
\be\begin{aligned}
S&=128\nop{LLL}+768\nop{LU_1U_2}+128\nop{U_1U_1U_1}+128\nop{U2U2U2}-264\sqrt{3}\nop{U_1(\partial^2 U_2)}\\&-264\sqrt{3}\nop{(\partial U_1)U_2}-264\sqrt{3}\nop{(\partial^2 L)L}+186\nop{(\partial L)(\partial L)}+372\sqrt{3}\nop{(\partial U_1)(\partial U_2)}+17\partial^4 L\\
S_1&=384\nop{LLU_1}+384\nop{LU_2U_2}+384\nop{U_1U_1U_2}-264\sqrt{3}\nop{L(\partial^2 U_1)}+372\sqrt{3}\nop{(\partial L)(\partial U_1)}\\&+186\sqrt{3}\nop{(\partial U_2)(\partial U_2)}-264\sqrt{3}\nop{(\partial^2 L)U_1}-264\sqrt{3}\nop{(\partial^2 U_2)U_2}+17\partial^4U_1\\
S_2&=384\nop{LLU_2}+384\nop{LU_1U_1}+384\nop{U_1U_2U_2}-264\sqrt{3}\nop{L(\partial^2 U_2)}+372\sqrt{3}\nop{(\partial L)(\partial U_2)}\\&+186\sqrt{3}\nop{(\partial U_1)(\partial U_1)}-264\sqrt{3}\nop{(\partial^2 L)U_2}-264\sqrt{3}\nop{(\partial^2 U_1)U_1}+17\partial^4U_2\end{aligned},\ee
and observe that $S,S_1,S_2$ are all weight 6 and transform in parallel to $L,U_1,U_2$ with respect to the $S_3$ action. From these parts we define the following $\mathbb{Z}_3$ invariant singular vectors 
\be\begin{aligned}\label{z3sings}
V_8^{\pm}&=\nop{U_1S_2}\pm\nop{U_2S_2}\\
V_9^{\pm}&=\nop{(\partial U_1)S_2}\pm\nop{(\partial U_2)S_2}\\
V_{10}^{\pm}&=\nop{(\partial^2 U_1)S_2}\pm\nop{(\partial^2 U_2)S_2}\\
Q_{10}^{\pm}&=\nop{U_1U_1S_1}\pm\nop{U_2U_2S_2}.\end{aligned}\ee
Now, we have everything ready to prove our first result of this section.
\begin{thm}\label{chalfz3}The simple orbifold $(\L_{\frac12}^{\otimes 3})^{\ZZ_3}$ is of type 2,4,5,6,6,7,8,9 and is strongly generated by $L$, together with $W_4,W_5,W_6,W_7,W_8,W_9$, and $C_6^-$.
\end{thm}
\begin{proof}
This result from fairly routine calculations involving the generators described in \eqref{z3gens} as well as the singular vectors \eqref{z3sings}. For example, the equation
$$\begin{aligned}128C_6^+&=S+450\sqrt{3}W_6-768\nop{L,W_4}-128\nop{LLL}-186\sqrt{3}\nop{(\partial L)(\partial L)}\\&+264\sqrt{3}\nop{(\partial^2 L)L}-186\sqrt{3}\partial^2 W_4+(-17+75\sqrt{3})\partial^4 L,\end{aligned}$$
eliminating the need for $C_6^+$ from the strong generating set. Similar equations exist to remove the remaining superfluous strong generators. 
\end{proof}

Now we move our attention to the orbifold  $(\L_{\frac12}^{\otimes 3})^{S_3}$, which we may view as the $\mathbb{Z}_2$ orbifold of  $(\L_{\frac12}^{\otimes 3})^{\ZZ_3}$ where the additional nontrivial action is given by $U_1 \leftrightarrow U_2$. Due to \eqref{ws} and \eqref{cs} it is clear that 
$$W_{2m+4},C^+_{m+6}\in (\L_{\frac12}^{\otimes 3})^{S_3}$$
for all $m\geq 0$.
In fact, of the generators described in Theorem \ref{chalfz3}, we have $L,W_4,W_6,W_8\in (\L_{\frac12}^{\otimes 3})^{S_3}$ while $C_6^-\mapsto -C_6^-$, $W_5\mapsto -W_5$, $W_7\mapsto -W_7$, and $W_9\mapsto -W_9$ under the additional $\mathbb{Z}_2$ action. In fact, we can check that the fields $L,W_4,W_6,W_8$ close under OPE and thus form a subalgebra of $(\L_{\frac12}^{\otimes 3})^{S_3}$. Direct computation shows that OPEs of odd terms are also in the subalgebra generated by $L,W_4,W_6,W_8$. From this, it is a routine calculation to construct the appropriate relations that remove the need for generators of the form $\nop{(\partial^a W_{2m+1})(\partial^b W_{2n+1})}$ and $\nop{(\partial^a W_{2m+1})(\partial^b C^-_{6})}$  from the strong generating set. 

\begin{thm} \label{ope-12} The orbifold algebra $\left(\L_{\frac12}^{\otimes 3}\right)^{S_3}$ is strongly generated by the fields $L,W_4,W_6,W_8$ and is thus of type $2,4,6,8$.
 \end{thm}
 \begin{rem} Using OPE package \cite{T} we have also shown that $L,W_2,W_4,W_6$ generators in Theorem \ref{ope-12} can be replaces with another strong set of generators $L,\tilde{W}_4,\tilde{W}_6,\tilde{W}_8$,  where $\tilde{W}_i$ are primary vectors under $L$. Thus this orbifold algebra is a $W$-algebra. 
 \end{rem}

 Now, we move to classify this orbifold using the tools in \cite{KL} which requires we correct the weight 4 field to be primary which may be done by 
 $$\widehat{W}_4=\mu(W_4-\frac{44}{59}\nop{LL}-\frac{9}{118}\partial^2L).$$
 where $\mu$ is a parameter to be fixed later. In \cite{KL}, the universal two parameter algebra, $\mathcal{W}^{\text{ev}}(c,\lambda)$, of type $\mathcal{W}(2,4,6,\dots)$ was rigorously constructed. This algebra is strongly generated by infinitely many fields in weights $2,4,6,\dots$ and weakly generated by a primary weight 4 field which we denote by $W_{\infty}^4$. We consider this algebra with central charge $\frac{3}{2}$ to correspond with the central charge of our orbifold. A normalization can be chosen for this field so that
\be\label{theirprod}
\left(W_{\infty}^4\right)_{(3)}W_{\infty}^4=816\lambda W_{\infty}^4-\frac{1088}{21} \left(-\frac{2303 \lambda ^2}{4}-1\right)\nop{LL}+\frac{680}{147} \left(-\frac{2303 \lambda ^2}{4}-1\right)2\partial^2 L.\ee
We also calculate
\be\label{ourprod}(\widehat{W}_4)_{(3)}\widehat{W}_4=-\frac{231}{118}\mu\widehat{W}_4+\frac{18552}{3481}\mu^2\nop{LL}-\frac{6615}{13924}\mu^2\partial^2L.\ee
Equating \eqref{theirprod} and \eqref{ourprod} gives $\lambda=\frac{22}{2891}$ and $\mu=-\frac{1088}{343}$, which does not match any other known algebras of type $(2,4,6,..,N)$, for instance the $\mathbb{Z}_2$ orbifold of the $\mathfrak{sl}_2$ parafermion algebra $(N_{10}(\mathfrak{sl}_2))^{\mathbb{Z}_2}$  has a $\lambda$ value of $\lambda=\frac{11}{272}$. The simple affine $\mathcal{W}$-algebras of type $B$ and $C$ as well as the $\mathbb{Z}_2$ orbifold of the affine $\mathcal{W}$-algebras of type $D$ are also of type $(2,4,6,...,N)$ but all have different values for $\lambda$ when $k$ is chosen so that their central charge is $\frac{3}{2}$.\color{black}

\section{The orbifold for $c=-\frac{22}{5}$ and affine $W$-algebras of type $G_2$}

In this part we study the simple $S_3$-orbifold algebra of $\L_{-\frac{22}{5}}$. We also consider two affine $W$-algebras associated to the exceptional rank two Lie algebra of type $\frak g_2$. 
There are four (nontrivial) nilpotent orbits of $\frak g_2$: short, long, subregular, and regular \cite{Coll}. In this paper, we consider the regular and subregular and the corresponding affine $W$-algebras. For more about affine $W$-algebras see \cite{KW}; see also \cite{F} for more about $W$-algebras of rank $2$.
As usual $W^k(\frak g_2,f)$ will denote the universal affine $W$-algebra of type $G_2$ and level $k$ associated to the nilpotent element $f$. Of course, two affine algebras are 
isomorphic if nilpotent elements come from the same nilpotent orbit. The unique simple quotient of $W^k(\frak g_2,f)$ will be 
denoted by $W_k(\frak g_2,f)$. 

\subsection{$W$-algebra $W^k(\frak g_2,f_{sub})$}

We choose realization of $\frak g_2$ using $8 \times 8$ matrices as in say \cite{H}. We also choose a nilpotent element $f_{sub}$ as in loc.cit. For this element $f:=f_{sub}$, let $\{e,f,h \}$ denote the corresponding $\frak{sl}_2$ triple.  Then with respect to ${\rm ad}(x)$, where $x=\frac{h}{2}$,  we have decompositions 
$$\frak{g}=\sum_{-2 \leq i \leq 2} \frak{g}_{i}$$
$$\frak{g}^f=\frak g_{-2}^f \oplus \frak g_{-1}^f$$ 
where $\frak g_{-2}^f$ is one-dimensional (this gives a generator of conformal weight $3$ inside the $W$-algebra) and $g_{-1}^f$ is $3$-dimensional. According to Kac-Wakimoto \cite{KW}  work 
$W^k(\frak g_2,f_{sub})$ is of type $(2^3,2)$. We denote the 
conformal generator of weight two by $L(z)$, with $E(z)$ and $F(z)$ two primaries of weight two, and by $G(z)$ the primary generator of conformal weight $3$. Then using the OPE 
program \cite{T} it is not difficult to obtain relations among generators. As far as we know, these OPE relations first appeared in J. Fasquel's PhD thesis \cite{F}.
\begin{prop} \label{OPE.sub} The following OPEs hold (we omit OPEs among $L(z)$ and $E(z),F(z), G(z)$ as those are uniquely determined) :
{
\begin{align*}
&E(z)E(w)  \sim  (10+3k) \frac{(4+k)c}{2(z-w)^4}+\frac{1}{(z-w)^2}  \left(2 (4+k)(10+3k)L(w)-4(3+k)F(w)\right) \\  & +\frac{1}{(z-w)}(  (4+k)(10+3k) (\partial L)(w) -2(3+k) (\partial E)(w)) \\
&F(z)F(w)  \sim  - (10+3k) \frac{(4+k)c}{2(z-w)^4} +\frac{1}{(z-w)^2} \left(-2 (4+k)(10+3k)L(w)-4(3+k)E(w)\right) \\ & +\frac{1}{(z-w)}( - (4+k)(10+3k) (\partial L)(w) -2(3+k) (\partial E)(w)) \\
&E(z) F(w)  \sim \frac{4(3+k)}{(z-w)^2} F(w)+\frac{1}{(z-w)} \left(-2 G(w)+2(3+k) (\partial F)(w) \right) \\
& G(z)G(w)  \sim \frac{(2+k)(10+3k)(16+5k)(4+k)c}{(z-w)^6} - \frac{3(2+k)(4+k)(10+3k)(16+5k)L(w)}{(z-w)^4} \\ & -\frac{3(2+k)(4+k)(10+3k)(16+5k)}{2(z-w)^3} \partial L(w) \\
&+ \frac{1}{(z-w)^2} \biggl(-(8+3k) \nop{E(w)^2} +2 (4+k)^2(10+3k) \nop{L(w)^2}+(8+3k) \nop{ F(w)^2} \\ & -\frac{3(2+k)(4+k)(8+3k)(10+3k)}{4} (\partial^2 L)(w) \biggr) \\
& + \frac{1}{(z-w)} (-(8+3k) \nop{E(w) (\partial E)} + 2(4+k)^2(10+3k) \nop{ L(w) (\partial L)(w)}  \\ & +(8+3k) \nop{F(w) \partial F(w)}  -\frac{(2+k)(4+k)(4+3k)(10+3k)}{6} \partial^3 L(w) ) \\
& E(z)G(w) \sim \frac{2(2+k)(16+5k)}{(z-w)^3} F(w)+\frac{(2+k)(16+5k)}{2(z-w)^2} \partial F(w)  \\ & +\frac{1}{z-w} \left(2 \nop{ E(w)F(w)}- 2(4+k) \nop{ L(w) F(w)} + 2 \partial G(w) +\frac{(2+k)^2}{2} 
\partial^2 F(w) \right) \\
& F(z)G(w) \sim \frac{2(2+k)(16+5k)}{(z-w)^3} E(w)+\frac{(2+k)(16+5k)}{2(z-w)^2} \partial E(w) \\ &  +\frac{1}{z-w} \left(-  \nop{ E(w)^2} - 2(4+k) \nop{L(w) E(w)} -\nop{ F(w)^2} +\frac{(2+k)^2}{2} 
\partial^2 E(w) \right).
\end{align*} 
}
where the central charge is $c=-\frac{4(k+2)(6k+17)}{k+4}$.
\end{prop}
Next we specialize $k=-\frac{16}{5}$ in Proposition \ref{OPE.sub}. We consider the ideal $I= \langle G \rangle$ generated by the primary element of degree $3$. From the OPEs for the quadratic 
generators, inside $W:=W^{-\frac{16}{5}}(\frak g_2,f_{sub})/I$ we can define three conformal vectors:
\begin{align*}
L_1(z) & =-\frac{5}{12} E(z)-\frac{5 i}{4 \sqrt{3}} F(z)+\frac{1}{3} L(z), \\
L_2(z) & =-\frac{5}{12} E(z)+\frac{5 i}{4 \sqrt{3}} F(z)+\frac{1}{3} L(z), \\
L_3(z) &  =\frac{5}{6} E(z)+\frac{1}{3} L(z),
\end{align*}
which mutually commute in $W$ (i.e. $L_i(z) L_j(w) \sim 0$, $i \neq j$) and each has central charge $c=-\frac{22}{5}$. Therefore $W$ must be a quotient of $\mathcal{V}_{-22/5}^{\otimes^3}$.
It is easy to see now that the maximal ideal $I_{max} \subset  W^{-\frac{16}{5}}(\frak g_2,f_{sub})$ must contain singular vectors $v_{sing}^{(i)}$, $i=1,2,3$ of degree $4$ for each of the three copies $\mathcal{V}_{-22/5}$.
We conclude that $I_{max}=\langle G, v_{sing}^{(1)},v_{sing}^{(2)},v_{sing}^{(3)}, \rangle $ and therefore 
\begin{prop} \label{subreg} We have an isomorphism of simple vertex operator algebras
$$W_{-\frac{16}{5}} (\frak g_2, f_{sub}) \cong \mathcal{L}_{-22/5}^{\otimes^3}.$$
\end{prop}

\subsection{$W$-algebra $W^k(\frak g_2,f_{prin})$} In this part we construct  the principal $W$-algebra of type $\frak g_2$.
The ${W}$-algebra ${W}^k(\frak g_2,f_{prin})$ is known to be of type $(2,6)$, where $\omega$ is just the conformal vector given in \cite{KW}. 
It is convenient to use the standard parametrization of the level $k=-h^\vee+\frac{p}{q}=-4+\frac{p}{q}$, so that the central charge is $c(k)=-\frac{2(12p-7q)(7p-4q)}{pq}$.
We assume here that 
$$\left(336 k^2+2301 k+3940\right) \left(588 k^2+3991 k+6752\right) \neq 0.$$
If $k$ one of the four roots of this polynomial (then $k$ is generic) one has to adjust the weight $6$ generator appropriately. We omit this computation here.

Using the standard approach, we have constructed this algebra leading to the following nonzero OPEs for the weight 6 generator, $W$, with itself.
The OPEs between $L(z)$ and $W(z)$ are clear, and we have 
$$W(z) W(w) \sim \sum_{n \geq 0}^{11} \frac{W_{(n)} W(w)}{(z-w)^{n+1}}.$$
Below we give explicit formulas for all the summands. As far as we know these formulas did not appear in the literature.

{
\begin{dmath*}
W_{(11)}W=-\frac{2p_0(k)}{81} (7 k+24) (12 k+41) \left(336 k^2+2301 k+3940\right) \left(588 k^2+3991 k+6752\right)\mathbb{1}
\end{dmath*}
\begin{dmath*}
W_{(9)}W=\frac{4p_0(k)}{27} (k+4) \left(336 k^2+2301 k+3940\right) \left(588 k^2+3991 k+6752\right)L
\end{dmath*}
\begin{dmath*}
W_{(8)}W=\frac{2p_0(k)}{27} (k+4) \left(336 k^2+2301 k+3940\right) \left(588 k^2+3991 k+6752\right)\partial L
\end{dmath*}
\begin{dmath*}
W_{(7)}W=\frac{p_1(k)}{108}\left(-62 (k+4)\nop{LL}+3 \left(84 k^2+579 k+1000\right)\partial^2L\right)
\end{dmath*}
\begin{dmath*}
W_{(6)}W=\frac{p_1(k)}{108}\left(-62 (k+4)\nop{(\partial L)L}+\frac{2}{3} \left(84 k^2+579 k+1000\right)\partial^3L\right)
\end{dmath*}
\begin{dmath*}
W_{(5)}W=\frac{280p_3(k)}{9}W+p_2(k)\Lambda^5(L)
\end{dmath*}
\begin{dmath*}
W_{(4)}W=\frac{140p_3(k)}{9}\partial W+p_2(k)\Lambda^4(L)
\end{dmath*}
\begin{dmath*}
W_{(3)}W=p_5(k)\left(\frac{10}{9} \left(588 k^2+3929 k+6504\right)\partial^2W-\frac{1240}{3} (k+4)\nop{LW}\right)+p_4(k)\Lambda^3(L)
\end{dmath*}
\begin{dmath*}W_{(2)}W=p_5(k)\Omega^2(W,L)+p_4(k)\Lambda^2(L)   \end{dmath*}
   {\small
 \begin{dmath*}
 W_{(1)}W=p_7(k)\Omega^1(W,L)+p_6(k)\Lambda^1(L)   \end{dmath*}
}   
   \begin{dmath*}
   W_{(0)}W=p_7(k)\Omega^0(W,L)+p_6(k)\Lambda^0(L)
 \end{dmath*}}
where 
\begin{align*}
p_0(k)=&(k+4) (2 k+5) (2 k+7) (3 k+10) (7 k+22) (7 k+23) (8 k+27) (9 k+34)(11 k+40)\\ & (12 k+37) (15 k+52) (15 k+53) (18 k+65)\\
p_1(k)=& (k+4)^2 (2 k+5) (2 k+7) (3 k+10) (7 k+22) (8 k+27) (9 k+34) (11 k+40)\\ & (12 k+37) (15 k+52) (18 k+65) \left(336 k^2+2301 k+3940\right) \left(588 k^2+3991
   k+6752\right)\\
p_2(k)=&(k+4)^2 (2 k+5) (2 k+7) (3 k+10) (7 k+22) (8 k+27) (9 k+34) (11 k+40)\\ & (12 k+37) (15 k+52) (18 k+65)\\
p_3(k)=&(k+4) (2 k+7) (3 k+10) (7 k+20) (12 k+35) (13 k+48) (24 k+89)\\& \left(3 k^2+24 k+47\right) \left(336 k^2+2301 k+3940\right) \left(588 k^2+3991 k+6752\right)\\
p_4(k)=&(k+4)^2 (2 k+5) (2 k+7) (3 k+10) (9 k+34) (11 k+40) (12 k+37)\\
p_5(k)=&(k+4) (2 k+7) (3 k+10) (7 k+20) (24 k+89) \left(3 k^2+24 k+47\right)\\& \left(336 k^2+2301 k+3940\right) \left(588 k^2+3991 k+6752\right)\\
p_6(k)=&(k+4)^2 (2 k+5) (2 k+7) (3 k+10) (9 k+34)\\
p_7(k)=&(k+4) \left(3 k^2+24 k+47\right) \left(336 k^2+2301 k+3940\right) \left(588 k^2+3991 k+6752\right),
\end{align*}
and the fields $\Lambda^n(L)$ only depend on $L$ while every summand in the fields $\Omega^n(W)$ depends on $W$.

\begin{rem}
In addition to constructing $\mathcal{W}^k(\mathfrak{g_2},f_{\text{princ}})$ directly using quantum Hamiltonian reduction, we in parallel showed that for generic central charge $c$ there is a unique universal $\mathcal{W}(2,6)$ algebra. This has been studied previously in the physics literature \cite{BFKNRV}.

\end{rem}

With these explicit formulas in hand, we are now ready to prove the main result of this section, also stated in the introduction.
\begin{thm} We have an isomorphism of rational vertex algebras:

$$W_{-\frac{16}{5}} (\frak g_2, f_{sub})^{S_3} \cong W_{-\frac{19}{6}} (\frak g_2, f_{prin}).$$

\end{thm} 
\begin{proof} Using Proposition \ref{subreg}, and explicit generators of weights $2$ and $6$ of $(\L_{-22/5}^{\otimes 3})^{S_3}$ obtained in \cite{MPS}, we can explicitly 
compute OPEs among generators. 
Computer computation with the OPE package \cite{T} shows that we get identical OPEs also from $W^{-\frac{19}{6}}(\frak g, f_{prin})$ (we only have to slightly normalize the $W$ generator given above). Using the universal property for $W^k(\frak{g},f_{prin})$, we get a vertex algebra map from $W^{-\frac{19}{6}}(\frak g, f_{prin})$ to $W_{-\frac{16}{5}} (\frak g_2, f_{sub})^{S_3}$. Since the $S_3$-orbifold of a simple vertex algebra is always simple \cite{CM}, we see that 
this map factors through the simple quotient $W_{-\frac{19}{6}}(\frak g, f_{prin})$.
\end{proof} 

From the OPEs we can also see that there are also collapsing levels to the Virasoro algebra. Here we slightly abuse the term "collapsing level" originally introduced 
in an important work of Adamovi\'c et al \cite{AKMPP} to indicate those levels for which the simple minimal affine $W$-algebra reduces to an affine vertex algebra. 

\begin{prop} The simple affine $W$-algebras $W_k(\frak g_2,f_{prin})$ collapses to a (simple) Virasoro vertex algebra if and only if $k \in \{  -\frac{34}{9}, -\frac{7}{2}, -\frac{10}{3}, -\frac{5}{2}, -\frac{22}{7}, -\frac{65}{18},-\frac{40}{11},-\frac{37}{12},-\frac{27}{8}, -\frac{52}{15},-\frac{53}{15},-\frac{23}{7} \}$.
\end{prop}

\begin{proof}
We first notice that in order for $W_k(\frak g_2):=W_k(\frak g_2,f_{prin})$ to collapse to a (simple) Virasoro vertex algebra we must have $$p_0(k)(7 k+24) (12 k+41) \left(336 k^2+2301 k+3940\right) \left(588 k^2+3991 k+6752\right)=0.$$
If $k\in \{  -\frac{34}{9}, -\frac{7}{2}, -\frac{10}{3}, -\frac{5}{2} \}$ this is quite evident from the fact that every summand in every pole of the OPE of $W$ with itself will be a normally ordered multiple of $W$ or one of its derivatives. These are also the only levels for which this occurs for the generic values of the central charge.
 The remaining cases are more interesting because the relevant Virasoro algebras are minimal. We will start with the details of $k=-\frac{65}{18}$, which corresponds to a central charge of $c=-\frac{46}{3}$. In this case there is a singular vector of conformal weight 8 which is inherited from the subalgebra copy of $\L_{-46/3}$, namely
\begin{dmath*}V_8=-\frac{757693104671125}{2529990231179046912}\left(29160\nop{LLLL}-12960\nop{(\partial L)(\partial L)L}-42120\nop{(\partial^2 L)LL}+5670\nop{(\partial^2 L)(\partial^2L)}+720\nop{(\partial^3 L)(\partial L)}-3960\nop{(\partial^4 L)L}-139\partial^6 L\right).\end{dmath*}
Using the above notation we have
\be\begin{aligned}\label{trickycollapse}
p_4\left(-\frac{65}{18}\right)\Lambda^3(L)&=V_8\\
p_4\left(-\frac{65}{18}\right)\Lambda^2(L)&=\partial V_8\\
p_6\left(-\frac{65}{18}\right)\Lambda^1(L)&=\frac{81}{820}\partial^2 V_8+\frac{78}{205}\nop{L V_8}\\
p_6\left(-\frac{65}{18}\right)\Lambda^0(L)&=-\frac{15}{41}\nop{(\partial L)V_8}+\frac{27}{82}\nop{L(\partial V_8)},
\end{aligned}\ee
and as such $L$ is not inside  the ideal generated by $W$, meaning ${W}_{-\frac{65}{18}}(\frak g_2)\cong \L_{-46/3}$.
The cases when $k=-\frac{22}{7}$ also corresponds to $c=-\frac{46}{3}$ and gives ${W}_{-\frac{22}{7}}(\frak g_2)\cong \L_{-46/3}$  similarly. Further the cases when $k=-\frac{27}{8}$ and $k=-\frac{65}{18}$ correspond to $c=-\frac{3}{5}$ for which there is also a Virasoro singular vector of conformal weight 8 leading to a set of equations similar to \eqref{trickycollapse}, yielding 
$$ {W}_{-\frac{27}{8}}(\frak g_2)\cong \L_{-3/5} \text{ and } {W}_{-\frac{65}{18}}(\frak g_2)\cong \L_{-3/5}.$$
The cases with $k=-\frac{40}{11}$ and $k=-\frac{37}{12}$ correspond to $c=-\frac{232}{11}$, where we have a singular vector of weight 10. In each of these cases $\Lambda^1(L)$ and $\Lambda^0(L)$ are multiples of the singular vector and its derivative, respectively, and thus 
$$ {W}_{-\frac{40}{11}}(\frak g_2)\cong \L_{-232/11}  \text{ and } {W}_{-\frac{37}{12}}(\frak g_2)\cong \L_{-232/11}.$$
Finally, the cases with $k=-\frac{53}{15}$ and $k=-\frac{23}{7}$ correspond to $c=-\frac{22}{5}$, for which their is a singular vector of weight 4 leading to equations similar to \eqref{trickycollapse} for the appropriate poles, thus
$$ {W}_{-\frac{53}{15}}(\frak g_2)\cong \L_{-22/5}  \text{ and } {W}_{-\frac{23}{7}}(\frak g_2)\cong \L_{-22/5}.$$ \color{black}
\end{proof}
\begin{rem} The above result  also gives a new proof or rationality of four admissible $\frak{g}_2$ minimal models: $(p,q)=(5,8)$, $(7,18)$, $(7,15)$ and $(4,11)$. 
Using asymptotic properties of characters of $\frak{g}_2$ and Virasoro minimal models, it follows that only two additional $G_2$ models 
are extensions of Virasoro minimal models: $(p,q)=(4,13)$ and $(5,11)$ of level $-\frac{48}{13}$ and $-\frac{39}{11}$. For the former we obtain decomposition
$$W_{-\frac{48}{13}}(\frak g_2) = L_{Vir}(c_{3,26},0) \oplus L_{Vir}(c_{3,26},6),$$
and for the latter we can write
$$W_{-\frac{39}{11}}(\frak g_2) = L_{Vir}(c_{11,30},0) \oplus M,$$
where $M$ is a $L(c_{11,30},0)$-module. Conjecturally, we expect 
$$W_{-\frac{39}{11}}(\frak g_2) = L_{Vir}(c_{11,30},0)  \oplus L_{Vir}(c_{11,30},6) \oplus L_{Vir}(c_{11,30},24) \oplus L_{Vir}(c_{11,30},63).$$

\end{rem}

The appearance of the group $S_3$ in the setup of $W^k(\frak g_2,f_{sub})$  is not a coincidence. It is known that 
any element of the component group $A(f)$ of a nilpotent orbit $\mathcal{O}$ induces an automorphism of the affine $W$-algebra.
Component and fundamental groups are always finite and their complete list can be found in \cite[Chapter 8]{Coll}.
In particular for the subregular orbit $\mathcal{O}_{sub}$ of $\frak g_2$ we have $\pi(\mathcal{O}_{sub})=A(f_{sub})=S_3$. It is also easy to show using explicit OPEs in Proposition \ref{OPE.sub} that 
${\rm Aut}(W^k(\frak g_2,f_{sub}))=S_3$.
There are a few more examples of nilpotent orbits of simple Lie algebra that are conjecturally related to permutation orbifolds of $(2,5)$-minimal models, the largest being $S_5$ for a particular nilpotent orbit $E_8(a_7)$ of $E_8$.  Let $f_{s_5} \in E_8(a_7)$. 
Motivated by numerical evidence we expect that 
\begin{conjecture}  We have an isomorphism 
$$W_{k}(\frak e_8,f_{s_5}) \cong \mathcal{L}_{-22/5}^{\otimes^5}$$
and moreover
$$W_{k}(\frak e_8,f_{s_5})^{S_5} \cong W_{-\frac{144}{5}}(\frak e_8,f_{sub})$$
where $k$ is a certain level such that $c(k)=-22$.
\end{conjecture}
The vertex algebra $W^{k}(\frak e_8,f_{sub})$ appeared in the work of Arakawa and van Ekeren, where they determined its type \cite{AE}.
Another example coming from the principal affine $W$-algebras of type  $\frak f_4$ was discussed in \cite{MPS}.
\newpage 

\section{Appendix}

\begin{dmath*}
\Lambda^0(L)=\left(\frac{5}{108} (k+4)^4 \left(6400 k^2+44499 k+77324\right) \left(15552 k^2+104313 k+174820\right)\nop{(\partial L)LLLL}-\frac{5}{27} (k+4)^3 \left(307988352 k^6+6283922508 k^5+53396773281 k^4+241877304092 k^3+616017723435 k^2+836341686788 k+472884439520\right)\nop{(\partial L)(\partial L)(\partial L)L}-\frac{10}{27} (k+4)^3 (275123520 k^6+5633364834 k^5+48051261273 k^4+218548121907 k^3+559006760767 k^2+762411306554 k+433163811960)\nop{(\partial^2L)(\partial L)LL}-\frac{5}{324} (k+4)^3 \left(969034752 k^6+19831695840 k^5+169065214212 k^4+768473937261 k^3+1964285118664 k^2+2677025069232 k+1519701610880\right)\nop{(\partial^3 L)LLL}+\frac{5}{144} (k+4)^2 \left(5638139136 k^8+154059566976 k^7+1841500124352 k^6+12576741811284 k^5+53677845043923 k^4+146605941466828 k^3+250228450849072k^2+244023285373120 k+104099884473600\right)\nop{(\partial^2L)(\partial^2L)(\partial L)}+\frac{5}{648} (k+4)^2 \left(18734696064 k^8+510397850592 k^7+6081840989784 k^6+41400733913190 k^5+176094016747663 k^4+479229337653394k^3+814897731901776k^2+791597177469376 k+336326728957440\right)\nop{(\partial^3 L)(\partial L)(\partial L)}+\frac{5}{648} (k+4)^2 \left(22340199168 k^8+610210770624 k^7+7291072985472 k^6+49774197908628 k^5+212341512557497 k^4+579671163698668 k^3+988878536674352k^2+963826243416704 k+410925607848960\right)\nop{(\partial^3 L)(\partial^2L)L}+\frac{5}{648} (k+4)^2 \left(13245410304 k^8+361228956480 k^7+4308954420744 k^6+29363893309686 k^5\\+125032199800487 k^4+340637406241986 k^3+579856576183872k^2+563876243190880 k\\+239824222944000\right)\nop{(\partial^4 L)(\partial L)L}+\frac{1}{648} (k+4)^2 \left(9845660160 k^8+268414896960 k^7+3200584374936 k^6+21801802242996 k^5+92791755994681 k^4+252682575276940 k^3+429917282396168k^2+417844424302784 k+177612968146560\right)\nop{(\partial^5 L)LL}\\-\frac{5}{3888}\left((k+4) \left(70157187072 k^{10}+2395569012480 k^9+36804815716032 k^8+335044043053056 k^7\\+2001297244970172 k^6+8196068493288525 k^5+23306444396904280k^4+45438751798882656 k^3\\+58127345062469696 k^2+44057798753256960 k+15024744784588800\right)\right)\nop{(\partial^4L)(\partial^3 L)}-\frac{1}{432} (k+4) \left(26244933120 k^{10}+896077256256 k^9+13765833642960 k^8+125302104393228 k^7+748383139429071 k^6\\+3064580827705257 k^5+8713479609911801
   k^4+16985912403534748 k^3+21726334014777568 k^2+16465260149252096 k+5614212652439040\right)\nop{(\partial^5L)(\partial^2 L)}-\frac{1}{7776} (k+4) \left(203131514880 k^{10}+6923867537664 k^9+106181890784928 k^8+964771316683128 k^7+5751509501236356 k^6\\+23506873565646933 k^5+66704618215297200k^4+129768328341307800 k^3+165637043357614976 k^2+125258475185645440 k+42615931244544000\right)\nop{(\partial^6 L)(\partial L)}-\frac{1}{54432} (k+4) \left(322620641280 k^{10}+10994200616448 k^9+168561180996768 k^8+1531141783523712 k^7+9125319182583162 k^6+37284399685556705 k^5+105765664885942706k^4+205685698522399232 k^3+262439463665615264 k^2+198382995639216128 k+67465854311966720\right)\nop{(\partial^7 L)L}+\frac{1}{7838208}\left(6022251970560 k^{12}+246417645809664 k^{11}+4620640231113984 k^{10}+52502905368301248 k^9+402625978586303472 k^8+2195280019747600740 k^7+8726413019732753489k^6+25481066563320502308 k^5+54244347536596270928 k^4+82102378453675558848 k^3+83866072140917341952 k^2+51910726542998341632 k+14724215780641198080\right)\partial^9L\right)
\end{dmath*}

\begin{dmath*}
\Lambda^1(L)=\left(\frac{1}{54} (k+4)^4 \left(6400 k^2+44499 k+77324\right) \left(15552 k^2+104313 k+174820\right)\nop{LLLLL}-\frac{1}{108} (k+4)^3 \left(18396661248 k^6+375252938352 k^5+3187752964032 k^4+14435363779845 k^3+36751413129532 k^2+49876565020640 k+28189138886400\right)\nop{(\partial L)(\partial L)LL}-\frac{1}{108} (k+4)^3 \left(7309080576 k^6+149628924864 k^5+1276015823004 k^4+5802191762295 k^3+14836963528264 k^2+20229656112560 k+11489725824000\right)\nop{(\partial^2L)LL}+\frac{1}{36} (k+4)^2 \left(23462695872 k^8+639247348656 k^7+7617735370116 k^6+51859882602369 k^5+220598608315643 k^4+600396725500187 k^3+1021026239675080
   k^2+991925566014784 k+421482454671360\right)\nop{(\partial^2 L)(\partial L)(\partial L)}+\frac{1}{18} (k+4)^2 \left(6994584576 k^8+191067584400 k^7+2283143277456 k^6+15587734929978 k^5+66504839400355 k^4+181569276291374 k^3+309776746880856
   k^2+301962382197456 k+128756373080640\right)\nop{(\partial^2L)(\partial^2L)L}+\frac{1}{324} (k+4)^2 \left(186037762176 k^8+5066868806496 k^7+60357766817448 k^6+410736997701210 k^5+1746416931870509 k^4+4750984123226478 k^3+8075489288448560
   k^2+7841219537106944 k+3329981663093760\right)\nop{(\partial^3L)(\partial L)L}+\frac{1}{324} (k+4)^2 \left(32936730624 k^8+898058750400 k^7+10710123170616 k^6+72967497795678 k^5+310615711789849 k^4+846002556795866 k^3+1439691140976208
   k^2+1399565091796448 k+595050159237120\right)\nop{(\partial^4L)LL}\\-\frac{1}{1944}\left((k+4) \left(491100309504 k^{10}+16750248391296 k^9+257051280481632 k^8+2337267620900808 k^7\\+13944382286929770 k^6  +57038097408972443 k^5+161993956733818346
   k^4+315431508176707400 k^3\\+403003320724401472 k^2+305066845875948032 k  +103900852253614080\right)\right)\nop{(\partial^3L)(\partial^3L)}\\-\frac{9}{2} (k+4) \left(175136919552 k^{10}+5979878037504 k^9+91868150287872 k^8+836252460625968 k^7+4994829802686456 k^6\\+20454391804109553 k^5+58160362836551484
   k^4+113382401413021648 k^3+145032681121579968 k^2+109918839350265344 k+37481595394344960\right)\nop{(\partial^4L)(\partial^2L)}-\frac{9}{10} (k+4) \left(435921937920 k^{10}+14849294253696 k^9+227573441700192 k^8+2066320856088936 k^7+12309641634769476 k^6+50273171379297971 k^5+142548444287558280
   k^4+277094185622399176 k^3+353390471229322752 k^2+267011039945734656 k+90762062273495040\right)\nop{(\partial^5L)(\partial L)}-\frac{1}{20} (k+4) \left(2019366236160 k^{10}+68818738157568 k^9+1055174614367328 k^8+9585390321860400 k^7+57131321802317178 k^6+233447722387774347
   k^5+662289994566698810 k^4+1288111802671584368 k^3+1643727302314709024 k^2+1242686605881161472 k+422671810809108480\right)\nop{(\partial^6 L)L}+\frac{3}{280} \left(1505562992640 k^{12}+61604411452416 k^{11}+1155160057778496 k^{10}+13125726342075312 k^9+100656494870524668 k^8\\+548820011579297193
   k^7+2181603341046966329 k^6+6370267278173246187 k^5+13561089829536455432 k^4+20525603316596375472 k^3+20966534092725067520 k^2+12977698548834862080
   k+3681061731402547200\right)\partial^8L\right)
\end{dmath*}

\begin{dmath*}
\Lambda^2(L)=\left(-\frac{1}{54} (k+4)^3 \left(92703744 k^4+1267774752 k^3+6500755699 k^2+14813187576 k+12656385200\right)\nop{(\partial L)LLL}+\frac{7}{324} (k+4)^2 \left(253925280 k^6+5192050608 k^5+44221218606 k^4+200812630609 k^3+512794652934 k^2+698174819888 k+395952508800\right)\nop{(\partial L)(\partial L)(\partial L)}+\frac{1}{108} (k+4)^2 \left(2120102208 k^6+43502274144 k^5+371903400636 k^4+1695590435761 k^3+4348182557408 k^2+5946567147312 k+3388333598400\right)\nop{(\partial^2L)(\partial L)L}+\frac{1}{162} (k+4)^2 \left(701338176 k^6+14385000672 k^5+122924343060 k^4+560169728659 k^3+1435743475008 k^2+1962381639248 k+1117447262400\right)\nop{(\partial^3 L)LL}-\frac{1}{162} (k+4) \left(2788079616 k^8+76356765072 k^7+914888731176 k^6+6263983859931 k^5+26804756946483 k^4+73409308592217 k^3+125651464323508 k^2+122897231016608
   k+52588383811200\right)\nop{(\partial^3 L)(\partial^2 L)}-\frac{1}{648} (k+4) \left(6614576640 k^8+180882084096 k^7+2163835759680 k^6+14790044447592 k^5+63175152837203 k^4+172683823446184 k^3+294973049842432
   k^2+287883548198528 k+122904465043200\right)\nop{(\partial^4 L)(\partial L)}-\frac{1}{3240}(k+4) \left(9850740480 k^8+269302852224 k^7+3220621636932 k^6+22006242372675 k^5+93966753547809 k^4+256756287752401 k^3+438412318910224 k^2+427697924147920
   k+182513991043200\right)\nop{(\partial^5L)L}\right)
\end{dmath*}

\begin{dmath*}
\Lambda^3(L)=\left(-\frac{1}{108} (k+4)^3 \left(92703744 k^4+1267774752 k^3+6500755699 k^2+14813187576 k+12656385200\right)\nop{LLLL}+\frac{1}{54} (k+4)^2 \left(1769433120 k^6+36170718192 k^5+307983344514 k^4+1398147789371 k^3+3569086504346 k^2+4857512624432 k+2753671656000\right)\nop{(\partial L)(\partial L)L}+\frac{1}{108} (k+4)^2 \left(2112058368 k^6+43328638080 k^5+370338214908 k^4+1688049810869 k^3+4327706491216 k^2+5916856032528 k+3370337692800\right)\nop{(\partial^2 L)LL}-\frac{1}{72} (k+4) \left(2788079616 k^8+76356765072 k^7+914888731176 k^6+6263983859931 k^5+26804756946483 k^4+73409308592217 k^3+125651464323508 k^2+122897231016608
   k+52588383811200\right)\nop{(\partial^2L)(\partial^2 L)}-\frac{1}{324} (k+4) \left(18545708160 k^8+506465186976 k^7+6050247604632 k^6+41294964740514 k^5+176132133532045 k^4+480724894384134 k^3+819917493355968
   k^2+798987612456896 k+340581640627200\right)\nop{(\partial^3 L)(\partial L)}-\frac{1}{1296}(k+4) \left(26316057600 k^8+719487788544 k^7+8605079033544 k^6+58802529192942 k^5+251108659932821 k^4+686196398950986 k^3+1171797687662880 k^2+1143279396494368
   k+487932447129600\right)\nop{(\partial^4 L)L}+\frac{1}{77760}\left(418211942400 k^{10}+14295878231040 k^9+219890262604320 k^8+2004122659390992 k^7+11986088867220990 k^6+49151648188457845 k^5+139958403544235118
   k^4+273252685532068176 k^3+350073420007844192 k^2+265747093350327552 k+90771065277696000\right)\partial^6 L\right)\end{dmath*}

\begin{dmath*}
\Lambda^4(L)=\left(\frac{5}{27} (k+4)^2 \left(23352 k^2+159815 k+273412\right)\nop{(\partial L)LL}-\frac{1}{12} (k+4) \left(204624 k^4+2806272 k^3+14435089 k^2+33007496 k+28308960\right)\nop{(\partial^2 L)(\partial L)}-\frac{1}{324} (k+4) \left(2441376 k^4+33467028 k^3+172062031 k^2+393206164 k+337000960\right)\nop{(\partial^3 L)L}+\frac{1}{432} \left(1185408 k^6+24378480 k^5+208905480 k^4+954774795 k^3+2454576540 k^2+3365465024 k+1922595712\right)\partial^5 L\right)\end{dmath*}

\begin{dmath*}
\Lambda^5(L)=\left(\frac{10}{81} (k+4)^2 \left(23352 k^2+159815 k+273412\right)\nop{LLL}-\frac{5}{108} (k+4) \left(613872 k^4+8372112 k^3+42798821 k^2+97197144 k+82739584\right)\nop{(\partial L)(\partial L)}-\frac{5}{108} (k+4) \left(733824 k^4+10060932 k^3+51734513 k^2+118250228 k+101370944\right)\nop{(\partial^2 L)L}+\frac{5}{324} \left(1185408 k^6+24378480 k^5+208905480 k^4+954774795 k^3+2454576540 k^2+3365465024 k+1922595712\right)\partial^4L\right)\end{dmath*}

\begin{dmath*}
\Omega^0(W,L)=\left(\frac{1}{36} \left(84672 k^6+1737288 k^5+14848530 k^4+67664823 k^3+173384186 k^2+236849032 k+134743680\right)\partial^5 W-\frac{8}{9} (k+4) \left(15246 k^4+204873 k^3+1029576 k^2+2292659 k+1908010\right)\nop{(\partial^3 L)W}-\frac{2}{9} (k+4) \left(140364 k^4+1892547 k^3+9550009 k^2+21373241 k+17898240\right)\nop{(\partial^2 L)(\partial W)}-\frac{1}{9} (k+4) \left(254520 k^4+3433662 k^3+17339635 k^2+38845538 k+32573520\right)\nop{(\partial L)(\partial^2 W)}-\frac{10}{9} (k+4) \left(9576 k^4+130128 k^3+662481 k^2+1497464 k+1267960\right)\nop{L(\partial^3 W)}+\frac{4}{3} (k+4)^2 \left(13728 k^2+93673 k+159620\right)\nop{(\partial L)LW}+\frac{2}{3} (k+4)^2 \left(13728 k^2+93673 k+159620\right)\nop{LL(\partial W)}\right)
\end{dmath*}

\begin{dmath*}
\Omega^1(W,L)=\left(\frac{1}{9} \left(336 k^2+2301 k+3940\right) \left(504 k^4+6804 k^3+34375 k^2+77012 k+64536\right)\partial^4 W-\frac{4}{3} (k+4) \left(46368 k^4+624357 k^3+3145493 k^2+7025839 k+5869260\right)\nop{(\partial^2 L)W}-\frac{2}{9} (k+4) \left(471240 k^4+6319242 k^3+31690355 k^2+70429798 k+58519920\right)\nop{(\partial L)(\partial W)}-\frac{10}{9} (k+4) \left(50400 k^4+678942 k^3+3422515 k^2+7650818 k+6398520\right)\nop{L(\partial^2 W)}+\frac{4}{3} (k+4)^2 \left(13728 k^2+93673 k+159620\right)\nop{LLW}\right)
\end{dmath*}

\begin{dmath*}
\Omega^2(W,L)=\left(\frac{10}{27} (16 k+57) (21 k+68)\partial^3 W-\frac{620}{3} (k+4)\nop{(\partial L)W}-\frac{620}{3} (k+4)\nop{L(\partial W)}\right)\end{dmath*}

\vspace{.3in}

\vspace{.2in}


\begin{thebibliography}{CalLM4}



 

\bibitem{A} T. Abe,  {\em $C_2$-cofiniteness of 2-cyclic permutation orbifold models}, Communications in Mathematical Physics 317.2 (2013): 425-445.

\bibitem{ALM} D. Adamovi\'c,  X. Lin,  and A. Milas, {\em ADE subagebras of the triplet vertex algebra: $A$-type}, Communications in Contemporary Mathematics, {\bf 15} (2013),1350028.

\bibitem{AKMPP} D. Adamovi\'c, D., Kac, V. G., Möseneder Frajria, P., Papi, P., and  Per\v{s}e, O.  {\em An application of collapsing levels to the representation theory of affine vertex algebras}, International Mathematics Research Notices, 2020(13), 4103-4143.

\bibitem{Ar} T. Arakawa, {\em Rationality of W-algebras: principal nilpotent cases}, Annals of Mathematics (2015): 565-604. 

\bibitem{Ar2} T. Arakawa, {\em Introduction to W-algebras and their representation theory}, In Perspectives in Lie theory (pp. 179-250). Springer, Cham.

\bibitem{AE} T. Arakawa and J. van Ekeren, {\em Rationality and fusion rules of exceptional W-algebras}, arXiv preprint arXiv:1905.11473 (2019).


\bibitem{BDM} K. Barron, C. Dong, and G. Mason, {\em Twisted sectors for tensor product vertex operator algebras associated to permutation groups},  Communications in Mathematical Physics 227.2 (2002): 349-384.


\bibitem{BHL} K. Barron, Y.-Z. Huang, and J. Lepowsky, {\em An equivalence of two constructions of permutation-twisted modules for lattice vertex operator algebras},  Journal of Pure and Applied Algebra 210.3 (2007): 797-826.

\bibitem{BFKNRV} R. Blumenhagen, M. Flohr, A. Kliem, W. Nahm, A. Recknagel, and R. Varnhagen, {\em W-algebras with two and three generators}, Advanced Series in Mathematical Physics | W-Symmetry, (1995): 235-269.




\bibitem{CM} S. Carnahan and M. Miyamoto, {\em Regularity of fixed-point vertex operator subalgebras},
preprint arXiv:1603.05645 (2016).

\bibitem{Coll} Collingwood, D. H., and McGovern, W. M.  {\em Nilpotent Orbits In Semisimple Lie Algebra: An Introduction}, 1993, CRC Press.




    
    
    
    
    
 
    

\bibitem{DLXY} C. Dong, H. Li, F. Xu, N. Yu, Fusion products of twisted modules in permutation orbifolds, arXiv:1907.00094.
  
\bibitem{DRX1} C. Dong,  L. Ren, and F. Xu. {\em On orbifold theory}, Advances in Mathematics 321 (2017): 1-30.
    
\bibitem{DRX2} C. Dong, F. Xu, and N. Yu. ,{\em The 3-permutation orbifold of a lattice vertex operator algebra},  Journal of Pure and Applied Algebra (2017).

\bibitem{DRX3} C. Dong, F. Xu, and N. Yu, {\em 2-permutations of lattice vertex operator algebras: Higher rank}, Journal of Algebra 476 (2017): 1-25.

\bibitem{DRX4} C. Dong, Chongying, F. Xu, and N. Yu. {\em S-matrix in permutation orbifolds}, to appear in Journal of Algebra (2022).
 
    



    








\bibitem{F} J. Fasquel, PhD thesis, 2022.

\bibitem{H} W. Hesselink, Indagationes Mathematicae, 30 (2019), 623-648.

\bibitem{JW} C. Jiang and Q. Wang. {\em Representations of $\mathbb{Z}_2$-orbifold of the parafermion vertex operator algebra $K (\frak{sl}_2, k)$} Journal of Algebra 529 (2019): 174-195.

\bibitem{KW} Kac, V.G. and Wakimoto, M., 2004. Quantum reduction and representation theory of superconformal algebras. Advances in Mathematics, 185(2), pp.400-458.

\bibitem{KL}  S. Kanade  and A. R. Linshaw, {\em Universal two-parameter even spin $W_\infty$ -algebra}, Advances in Mathematics 355 (2019): 106774.





\bibitem{HLi} H. Li, {\em Vertex algebras and vertex Poisson algebras} Communications in Contemporary Mathematics, 6 (2001), pp.61-110.


\bibitem{LMW} H. Li, A. Milas and J. Wauchope, {\em $S_2$-permutation orbifolds of $N=1$ and $N=2$ superconformal vertex algebras}, 


\bibitem{L2} A. Linshaw, {\em Invariant Theory and the Heisenberg Vertex Algebra},
{\em International Mathematics Research Notices}, Vol. 2012, No. 17, pp. 4014-4050

\bibitem{L3} A. Linshaw, {\em Invariant subalgebras of affine vertex algebras}, {\em Advances in Mathematics} {\bf 234}  (2013), 61-84.


\bibitem{MP2} A. Milas and M. Penn, {\em Permutation orbifolds of \(\frak {sl} _2\) vertex operator algebras}. Glasnik matematički 55, no. 2 (2020): 277-300.

\bibitem{MPS}  A. Milas, M. Penn, and C. Sadowski, {\em Permutation orbifolds of Virasoro vertex algebras and W-algebras}, Journal of Algebra 570 (2021): 267-296.


\bibitem{MPSh} A. Milas, M. Penn and H. Shao, {\em Permutation orbfolds of the Heisenberg vertex algebra, $\mathcal{H}(3)$}, Journal of Mathematical Physics, 60 (2019) 021703.

\bibitem{MPW}  A. Milas, M. Penn, and J. Wauchope, {\em Permutation orbifolds of rank three fermionic vertex superalgebras}, In Affine, Vertex and W-algebras (pp. 183-202). (2019), Springer, Cham.







\bibitem{T}  K. Theilman, Mathematica Package OPE, 1991.
  
 
  
  
  
\bibitem{W} H. Weyl, {\em The Classical Groups: Their Invariants and Representations},
  Princeton University Press, 1946.



















  




\end{thebibliography}
\end{document}